\documentclass[10pt]{article}

\usepackage{amsthm, amsmath, amsfonts, amssymb, latexsym}
\usepackage[utf8]{inputenc}
\usepackage[english]{babel}
\usepackage[hidelinks]{hyperref}
\usepackage{orcidlink}
\usepackage{graphicx}
\usepackage{xcolor}
\usepackage{pict2e}

\usepackage[backend=biber,style=trad-unsrt,doi=true,url=false,isbn=false]{biblatex}
\theoremstyle{plain}
\newtheorem{theorem}{Theorem}
\newtheorem{lemma}{Lemma}
\newtheorem{corollary}{Corollary}

\newtheorem{conjecture}{Conjecture}
\newtheorem{question}{Question}
\theoremstyle{definition}
\newtheorem{definition}{Definition}
\theoremstyle{remark}

\begin{document}

    \title{Large induced forests in planar multigraphs}
    \author{Mikhail Makarov\orcidlink{0009-0004-3604-4547}\footnote{mikhail.makarov.math@gmail.com}}
    \date{}
    \maketitle

    \begin{abstract}
        For a graph $G$, denote by $a(G)$ the number of vertices in the largest induced forest in $G$. The Albertson-Berman conjecture, which has been open since 1979, states that $a(G) \geq \frac{n}{2}$ for every simple planar graph $G$ on $n$ vertices. We show that the version of this problem for multigraphs (allowing parallel edges) is easily reduced to the problem about the independence number of simple planar graphs. Specifically, we prove that $a(M) \geq \frac{n}{4}$ for every planar multigraph $M$ and that this lower bound is tight. Then, we study the case when the number of pairs of vertices with parallel edges, which we denote by $k$, is small. In particular, we prove the lower bound $a(M) \geq \frac{2}{5}n-\frac{k}{10}$ and that the Albertson-Berman conjecture for simple graphs, assuming that it holds, would imply the lower bound $a(M) \geq \frac{n-k}{2}$ for multigraphs, which would be better than the general lower bound when $k$ is small. Finally, we study the variant of the problem where the plane multigraphs are prohibited from having $2$-faces, which is the main non-trivial problem that we introduce in this article. For that variant without $2$-faces, we prove the lower bound $a(M) \geq \frac{3}{10}n+\frac{7}{30}$ and give a construction of an infinite sequence of multigraphs with $a(M)=\frac{3}{7}n+\frac{4}{7}$.
    \end{abstract}

    \section{Introduction}
        For a graph $G$, denote by $a(G)$ the number of vertices in the largest induced forest in $G$.
        
        \begin{conjecture}[Albertson-Berman, AB, \cite{ab_conjecture}]\label{conjecture:ab}
            For every simple planar graph $G$ on $n$ vertices, we have $a(G) \geq \frac{n}{2}$.
        \end{conjecture}

        The AB conjecture has been open since 1979.

        Motivated by the AB conjecture, we consider the version of this problem for multigraphs, allowing parallel edges (but still not allowing loops). It turns out that it is significantly easier for multigraphs than for simple graphs, and it is easily reduced to the problem about the independence number of simple planar graphs. In Section~\ref{section:general_lower_bound_a}, we give a full solution to this problem for planar multigraphs. In Section~\ref{section:variants_reducing_to_independence_number_of_simple_planar_graphs}, we study some variants of the problem that are also easily reduced to problems about the independence number of simple planar graphs. Then, in Section~\ref{section:lower_bound_a_for_small_number_of_pairs_of_vertices_with_parallel_edges}, we obtain better lower bounds for the case when the number of pairs of vertices with parallel edges between them is small. Finally, in Section~\ref{section:without_2-faces}, we study the variant of the problem without $2$-faces, which is the main non-trivial problem that we introduce in this article.

    \section{The general lower bound and its tightness}\label{section:general_lower_bound_a}
        Denote by $\alpha(G)$ the independence number of $G$, which is the number of vertices in the largest independent set in $G$.

        \begin{lemma}\label{lemma:multigraph_to_simple_graph_a_alpha}
            For every planar multigraph $M$, there exists a simple planar graph $G_M$ with the same number of vertices such that $a(M) \geq \alpha(G_M)$.
        \end{lemma}
        \begin{proof}
            Consider the simple graph $G_M$ obtained from $M$ by deduplicating parallel edges (keeping a single edge for each pair of vertices with parallel edges). Clearly, $G_M$ is planar.
            
            Let $S$ be a set of vertices that is a maximum independent set in $G_M$, that is, $G_M[S]$ is an induced edgeless graph. Clearly, $S$ induces an edgeless graph in $M$ as well, that is, $F=M[S]$ is an induced edgeless graph in $M$, and hence $F$ is an induced forest in $M$. Therefore, we have $a(M) \geq |V(F)|=|S|=\alpha(G_M)$, as claimed.
        \end{proof}

        \begin{lemma}\label{lemma:simple_graph_to_multigraph_a_alpha}
            For every simple planar graph $G$, there exists a planar multigraph $M_G$ with the same number of vertices such that $a(M_G)=\alpha(G)$.
        \end{lemma}
        \begin{proof}
            Consider the multigraph $M_G$ obtained from $G$ by duplicating all edges.
            
            It is known that duplicating an edge preserves planarity~\cite[Proposition~7.3.1, p.~310]{gross_yellen_anderson_book_graph_theory}. Therefore, $M_G$ is planar.

            Any two adjacent vertices in $G$ induce a cycle of length $2$ in $M_G$. Therefore, any set of vertices inducing a forest in $M_G$ must induce an edgeless graph in $M_G$ and in $G$. Obviously, the converse also holds: any set of vertices inducing an edgeless graph in $G$ induces an edgeless graph in $M_G$ and hence a forest in $M_G$. Therefore, we have $a(M_G)=\alpha(M_G)=\alpha(G)$, as claimed.
        \end{proof}

        \begin{theorem}\label{theorem:planar_multigraphs_a_geq_n/4_tight}
            For every planar multigraph $M$ on $n$ vertices, we have $a(M) \geq \frac{n}{4}$. That lower bound is tight as there exists an infinite sequence of planar multigraphs with $a(M)=\frac{n}{4}$.
        \end{theorem}
        \begin{proof}
            By Lemma~\ref{lemma:multigraph_to_simple_graph_a_alpha}, there exists a simple planar graph $G_M$ with the same number of vertices $n$ such that $a(M) \geq \alpha(G_M)$. It is well-known that the lower bound $\alpha(G) \geq \frac{n}{4}$ holds for every simple planar graph $G$~\cite[Proposition~1]{bickle_independence_number_of_max_planar_graphs}. Therefore, we have $a(M) \geq \alpha(G_M) \geq \frac{n}{4}$, as claimed.

            To show that the lower bound $a(M) \geq \frac{n}{4}$ is tight, we prove that there exists an infinite sequence of planar multigraphs with $a(M)=\frac{n}{4}$. First, take an infinite sequence of simple planar graphs $\{G_k\}$ with $\alpha(G_k)=\frac{n}{4}$, that is, on which the lower bound $\alpha(G) \geq \frac{n}{4}$ is attained. It is known that such graphs exist~\cite[Section~1]{bickle_independence_number_of_max_planar_graphs}. Now, by Lemma~\ref{lemma:simple_graph_to_multigraph_a_alpha}, for each $k$, there exists a planar multigraph $M_k=M_{G_k}$ with the same number of vertices $n$ such that $a(M_k)=\alpha(G_k)=\frac{n}{4}$, which means that $\{M_k\}$ is the desired infinite sequence of planar multigraphs.
        \end{proof}

        In the proof of Theorem~\ref{theorem:planar_multigraphs_a_geq_n/4_tight}, we deliberately did not take any specific tight examples of simple planar graphs attaining $\alpha(G)=\frac{n}{4}$ in order to show that any such tight examples work. Now, in order to show an explicit construction of the tight infinite sequence of multigraphs $\{M_k\}$ attaining $a(M_k)=\frac{n}{4}$, we choose a specific sequence $\{G_k\}$. The standard choice for $G_k$ is a disjoint union of $k$ copies of $K_4$. Then, following the construction of $M_G$ in the proof of Lemma~\ref{lemma:simple_graph_to_multigraph_a_alpha}, we obtain that $M_k$ is a disjoint union of $k$ copies of $K_4$ with all edges duplicated. Other tight examples of simple planar graphs attaining $\alpha(G)=\frac{n}{4}$ are known~\cite[Section~1]{bickle_independence_number_of_max_planar_graphs}, and any of them would similarly produce an explicit tight example of a multigraph attaining $a(M)=\frac{n}{4}$.

        In the proof of Theorem~\ref{theorem:planar_multigraphs_a_geq_n/4_tight}, we used the lower bound $\alpha(G) \geq \frac{n}{4}$ for simple planar graphs. This lower bound is a corollary of the Four Color Theorem~\cite{appel_haken_4_color_theorem_part1, appel_haken_4_color_theorem_part2}. To derive this lower bound from the Four Color Theorem, we take the set of the vertices of the largest color class in any $4$-coloring of $G$. Furthermore, this is the only known proof of this lower bound; there is no known proof of this lower bound that does not use the Four Color Theorem~\cite[abstract]{cranston_rabern_planar_independence_number_3_13}. So, a natural question to ask is whether there is a proof of Theorem~\ref{theorem:planar_multigraphs_a_geq_n/4_tight} that avoids using the lower bound $\alpha(G) \geq \frac{n}{4}$ and the Four Color Theorem. We do not know such a proof, and the answer is most likely negative because our lower bound $a(M) \geq \frac{n}{4}$ from Theorem~\ref{theorem:planar_multigraphs_a_geq_n/4_tight} does not just follow from the lower bound $\alpha(G) \geq \frac{n}{4}$, but is in fact equivalent to it. To show that, we prove the implication in the other direction, that is, that the lower bound $a(M) \geq \frac{n}{4}$ implies the lower bound $\alpha(G) \geq \frac{n}{4}$. Take an arbitrary simple planar graph $G$ on $n$ vertices. By Lemma~\ref{lemma:simple_graph_to_multigraph_a_alpha}, there exists a planar multigraph $M_G$ with the same number of vertices $n$ such that $a(M_G)=\alpha(G)$. Now, if we apply the lower bound $a(M) \geq \frac{n}{4}$ to $M_G$, then we get $\alpha(G)=a(M_G) \geq \frac{n}{4}$. So, the lower bound $a(M) \geq \frac{n}{4}$ implies the lower bound $\alpha(G) \geq \frac{n}{4}$. This means that, if there is a proof of the lower bound $a(M) \geq \frac{n}{4}$ avoiding using the lower bound $\alpha(G) \geq \frac{n}{4}$ and avoiding using the Four Color Theorem, then, as an easy corollary of it, we would get a proof of the lower bound $\alpha(G) \geq \frac{n}{4}$ also avoiding using the Four Color Theorem. And, as we already said, no such proof of the lower bound $\alpha(G) \geq \frac{n}{4}$ is known.

    \section{The variants without triangles and for linear forests}\label{section:variants_reducing_to_independence_number_of_simple_planar_graphs}
        There are several variants of the original problem about induced forests in simple planar graphs that have been studied. In particular, the variant where graphs are restricted to be without triangles and the variant for linear forests. We show here that these variants for planar multigraphs are also easily reduced to the corresponding variants for the independence number of simple planar graphs.
        
        \begin{theorem}\label{theorem:planar_multigraphs_without_triangles_a_geq_n+1/3_tight}
            For every planar multigraph $M$ without triangles, we have $a(M) \geq \frac{n+1}{3}$. That lower bound is tight as there exists an infinite sequence of planar multigraphs with $a(M)=\frac{n+1}{3}$.
        \end{theorem}
        \begin{proof}
            It is easy to see that the constructions in Lemma~\ref{lemma:multigraph_to_simple_graph_a_alpha} and Lemma~\ref{lemma:simple_graph_to_multigraph_a_alpha} preserve the triangle-free property. So, similarly to the proof of Theorem~\ref{theorem:planar_multigraphs_a_geq_n/4_tight}, the problem about $a(M)$ for planar multigraphs without triangles is reduced to the problem about the independence number $\alpha(G)$ for simple planar graphs without triangles. It remains only to note that it is known~\cite[Proposition~2a, Proposition~2b, Remark on p.~293]{steinberg_tovey_planar_ramsey_numbers} that $\alpha(G) \geq \frac{n+1}{3}$ for every planar graph $G$ without triangles and that that lower bound is tight as there exists an infinite sequence of planar graphs without triangles with $\alpha(G)=\frac{n+1}{3}$ (in~\cite{steinberg_tovey_planar_ramsey_numbers}, the lower bound is proved as $\alpha(G) \geq \lfloor \frac{n}{3} \rfloor + 1$, but it is easy to see that $\lfloor \frac{n}{3} \rfloor + 1 = \lceil \frac{n+1}{3} \rceil$ for every integer $n$, so their lower bound is equivalent to $\alpha(G) \geq \frac{n+1}{3}$; this more convenient expression $\alpha(G) \geq \frac{n+1}{3}$ for the lower bound is also cited, for example, in~\cite[the beginning of the article]{dmmp_planar_graphs_without_triangles_with_smallest_independence_number}).
        \end{proof}

        Denote by $a_{\ell}(M)$ the number of vertices in the largest induced linear forest in $M$.
        
        \begin{theorem}\label{theorem:planar_multigraphs_al_geq_n/4_tight}
            For every planar multigraph $M$, we have $a_{\ell}(M) \geq \frac{n}{4}$. That lower bound is tight as there exists an infinite sequence of planar multigraphs with $a_{\ell}(M)=\frac{n}{4}$.
        \end{theorem}
        \begin{proof}
            It is easy to see that the statements and proofs of Lemma~\ref{lemma:multigraph_to_simple_graph_a_alpha} and Lemma~\ref{lemma:simple_graph_to_multigraph_a_alpha} still hold if we restrict forests to linear forests and replace $a(M)$ with $a_{\ell}(M)$. So, similarly to the proof of Theorem~\ref{theorem:planar_multigraphs_a_geq_n/4_tight}, the problem about $a_{\ell}(M)$ for planar multigraphs is reduced to the problem about the independence number $\alpha(G)$ for simple planar graphs.
        \end{proof}

    \section{The lower bound for a small number of pairs of vertices with parallel edges}\label{section:lower_bound_a_for_small_number_of_pairs_of_vertices_with_parallel_edges}
        In this section, we consider planar multigraphs with a small number of pairs of vertices with parallel edges and derive another lower bound on $a(M)$ from a lower bound on $a(G)$ for simple planar graphs. This new lower bound on $a(M)$ is better in that case than the general lower bound from Theorem~\ref{theorem:planar_multigraphs_a_geq_n/4_tight}.
        
        First, we observe that there is a straightforward, but suboptimal, derived lower bound for that case, which is as follows. Consider a planar multigraph $M$ on $n$ vertices and with $k$ pairs of vertices that have parallel edges between them. Assume that a lower bound $a(G) \geq a(n)$, where $a(n)$ is a function of $n$ that does not depend on $G$, holds for every simple planar graph $G$. Then deduplicate all parallel edges in $M$ as in the proof of Lemma~\ref{lemma:multigraph_to_simple_graph_a_alpha}, denote the resulting simple planar graph by $G_M$, apply to it the assumed lower bound for simple planar graphs, and get $a(G_M) \geq a(n)$. Then, there exists a set of vertices $S$ of size at least $a(n)$ that induces a forest in $G_M$. For each of the $k$ pairs of vertices with parallel edges in $M$, we remove one of these two vertices if both of them belong to $S$ and if neither of them was already removed in the previous steps for other pairs of vertices. In total, we remove at most $k$ vertices from $S$, and it is easy to see that the resulting set, which we denote by $S'$, induces a forest in $M$. Therefore, we have $a(M) \geq |S'| \geq |S|-k=a(G_M)-k \geq a(n)-k$. In particular, the AB conjecture for simple planar graphs, assuming that it holds, would imply the lower bound $a(M) \geq \frac{n}{2}-k$. Below, we prove another lower bound $a(M) \geq a(n+k)-k$, which is better than this straightforward lower bound $a(M) \geq a(n)-k$.

        \begin{lemma}\label{lemma:multigraph_to_simple_graph_subdividing_duplicate_edges_a}
            For every planar multigraph $M$ on $n$ vertices and with $k$ pairs of vertices that have parallel edges between them, there exists a simple planar graph $G_M$ on $n+k$ vertices such that $a(G_M)=a(M)+k$.
        \end{lemma}
        \begin{proof}
            For each of the $k$ pairs of vertices with parallel edges between them in $M$, first, we remove the parallel edges to reduce the multiplicity of the parallel edges to $2$, and then we subdivide one of the $2$ remaining parallel edges with a new vertex. Denote the resulting graph by $G_M$. Clearly, $G_M$ is simple because the edge removals and the edge subdivisions eliminate all existing parallel edges and no new parallel edges are created in the process. In order to show that $G_M$ is also planar, we observe that we can represent the procedure of the edge removals and the edge subdivisions as being done sequentially, one edge at a time. Each such operation individually preserves planarity: removing an edge obviously preserves planarity and subdividing an edge is well-known to preserve planarity~\cite[Proposition~7.5.15, p.~334 and Proposition~7.2.1, p.~306]{gross_yellen_anderson_book_graph_theory}. Since each step in this sequential procedure preserves planarity, the entire procedure does as well. The number of vertices in $G_M$ is clearly $n+k$.

            Take a maximum induced forest $F_M$ in $M$. Consider a pair of vertices $u$ and $v$ with parallel edges between them in $M$. Denote by $w$ the new added vertex from the subdivided edge between $u$ and $v$. The parallel edges form a $2$-cycle, so at most $1$ of the vertices $u$ and $v$ belongs to $F_M$. Then we can add $w$ to $F_M$ without creating any cycles. When we do this for all $k$ pairs of vertices with parallel edges between them in $M$, we get an induced forest in $G_M$, which we denote by $F_G$, that has $k$ more vertices than $F_M$. Therefore, $a(G_M) \geq |V(F_G)| = |V(F_M)|+k = a(M)+k$.

            Now, take a maximum induced forest $F_G$ in $G_M$. Consider a pair of vertices $u$ and $v$ with parallel edges between them in $M$. Denote by $w$ the new added vertex from the subdivided edge between $u$ and $v$. Since $uvw$ is a triangle in $G_M$, at most $2$ of these three vertices belong to $F_G$. If $w$ belongs to $F_G$, then at most one of $u$ and $v$ belongs to $F_G$, and we remove $w$. If both $u$ and $v$ belong to $F_G$, then $w$ cannot belong to $F_G$, and we arbitrarily remove one of $u$ and $v$ (if one of them was not already removed in previous steps for another pair of vertices). In all cases, after the removal of at most $1$ vertex, $w$ does not belong to the resulting forest and at most one of $u$ and $v$ belongs to the resulting forest. So, the parallel edges in $M$ between $u$ and $v$ cannot create a $2$-cycle in the resulting forest. When we do this for all $k$ pairs of vertices with parallel edges between them in $M$, we clearly get an induced forest in $M$, which we denote by $F_M$, and we removed at most $k$ vertices from $F_G$. Therefore, $a(M) \geq |V(F_M)| \geq |V(F_G)|-k=a(G_M)-k$.

            Combining the two obtained opposite inequalities $a(G_M) \geq a(M)+k$ and $a(M)+k \geq a(G_M)$, we get the equality $a(G_M)=a(M)+k$, as claimed.
        \end{proof}

        \begin{theorem}\label{theorem:lower_bound_for_simple_graphs_implies_lower_bound_for_multigraphs_for_small_number_of_pairs_of_vertices_with_parallel_edges}
            Any lower bound $a(G) \geq a(n)$, where $a(n)$ is a function of $n$ that does not depend on $G$, for every simple planar graph $G$ on $n$ vertices implies the lower bound $a(M) \geq a(n+k)-k$ for every planar multigraph $M$ on $n$ vertices and with $k$ pairs of vertices that have parallel edges between them.
        \end{theorem}
        \begin{proof}
            Consider a planar multigraph $M$ on $n$ vertices and with $k$ pairs of vertices that have parallel edges between them. By Lemma~\ref{lemma:multigraph_to_simple_graph_subdividing_duplicate_edges_a}, there exists a simple planar graph $G_M$ on $n+k$ vertices such that $a(G_M)=a(M)+k$. Applying the lower bound $a(G) \geq a(n)$ to $G_M$, we get $a(G_M) \geq a(n+k)$. Finally, we have $a(M)=a(G_M)-k \geq a(n+k)-k$, as claimed.
        \end{proof}

        \begin{corollary}\label{corollary:lower_bound_for_multigraphs_2n/5-3k/5}
            For every planar multigraph $M$ on $n$ vertices and with $k$ pairs of vertices that have parallel edges between them, we have $a(M) \geq \frac{2}{5}n-\frac{3}{5}k$.
        \end{corollary}
        \begin{proof}
            It is an application of Theorem~\ref{theorem:lower_bound_for_simple_graphs_implies_lower_bound_for_multigraphs_for_small_number_of_pairs_of_vertices_with_parallel_edges} to the best known lower bound $a(G) \geq \frac{2}{5}n$ for simple planar graphs, which follows from the existence of an acyclic $5$-coloring~\cite{borodin_acyclic_5-coloring} by taking the two largest color classes (an \emph{acyclic coloring} is defined as a proper coloring where the union of any two color classes induces a forest, or, equivalently, as a proper coloring where any cycle contains at least $3$ colors).
        \end{proof}

        \begin{corollary}\label{corollary:ab_conjecture_implies_lower_bound_for_multigraphs_n/2-k/2}
            The AB conjecture for simple planar graphs, assuming that it holds, would imply that, for every planar multigraph $M$ on $n$ vertices and with $k$ pairs of vertices that have parallel edges between them, we have $a(M) \geq \frac{n-k}{2}$.
        \end{corollary}
        \begin{proof}
            It directly follows from Theorem~\ref{theorem:lower_bound_for_simple_graphs_implies_lower_bound_for_multigraphs_for_small_number_of_pairs_of_vertices_with_parallel_edges}.
        \end{proof}

        In fact, the existence of an acyclic $5$-coloring for simple planar graphs implies a better lower bound on $a(M)$ than in Corollary~\ref{corollary:lower_bound_for_multigraphs_2n/5-3k/5} using a different argument than through the lower bound $a(G) \geq \frac{2}{5}n$ and Theorem~\ref{theorem:lower_bound_for_simple_graphs_implies_lower_bound_for_multigraphs_for_small_number_of_pairs_of_vertices_with_parallel_edges}.

        \begin{theorem}\label{theorem:lower_bound_for_multigraphs_2n/5-k/10}
            For every planar multigraph $M$ on $n$ vertices and with $k$ pairs of vertices that have parallel edges between them, we have $a(M) \geq \frac{2}{5}n-\frac{k}{10}$.
        \end{theorem}
        \begin{proof}
            Consider the simple planar graph $G_M$ obtained from $M$ by deduplicating parallel edges as in the proof of Lemma~\ref{lemma:multigraph_to_simple_graph_a_alpha}. Consider an acyclic $5$-coloring~\cite{borodin_acyclic_5-coloring} of $G_M$ and denote by $C_1$, $C_2$, $C_3$, $C_4$, $C_5$ the color classes of that coloring. For $1 \leq i < j \leq 5$, denote by $k_{ij}$ the number of pairs of vertices with one of the vertices in $C_i$, another one in $C_j$, that have parallel edges between them in $M$. Since the coloring is proper, there are no parallel edges (or any edges at all) inside any single color class. Therefore, each pair of vertices with parallel edges between them is counted in exactly one of the $k_{ij}$. Therefore, we have $\sum_{1 \leq i < j \leq 5}{k_{ij}}=k$. By definition of an acyclic coloring, the union of $C_i$ and $C_j$ for any $i$ and $j$ such that $1 \leq i < j \leq 5$ induces a forest in $G_M$. By removing one of the vertices from each of the $k_{ij}$ pairs of vertices that have parallel edges between them, we obtain an induced forest in $M$ with at least $|C_i|+|C_j|-k_{ij}$ vertices. Denote $a_{ij}=|C_i|+|C_j|-k_{ij}$. So, we have $a(M) \geq a_{ij}$ for all $i$ and $j$. Now, we choose the largest number among $a_{ij}$. To bound that largest number $\max_{i,j}(a_{ij})$ from below, we observe that, by the pigeonhole principle, the largest number is always greater than or equal to the average. So, we have
            \begin{multline*}
                a(M) \geq \max_{i,j}(a_{ij}) \geq \frac{1}{10}\sum_{1 \leq i < j \leq 5}{a_{ij}}=\frac{1}{10}\sum_{1 \leq i < j \leq 5}{(|C_i|+|C_j|-k_{ij})}=\\
                =\frac{1}{10}\left(\sum_{1 \leq i < j \leq 5}{(|C_i|+|C_j|)}-\sum_{1 \leq i < j \leq 5}{k_{ij}}\right)=\frac{1}{10}\left(\sum_{1 \leq i < j \leq 5}{(|C_i|+|C_j|)}-k\right)=\\
                =\frac{1}{10}\left(\frac{1}{2}\sum_{1 \leq i \leq 5, 1 \leq j \leq 5, i \neq j}{(|C_i|+|C_j|)}-k\right)=\\
                =\frac{1}{10}\left(\frac{1}{2}\sum_{1 \leq i \leq 5, 1 \leq j \leq 5, i \neq j}{|C_i|}+\frac{1}{2}\sum_{1 \leq i \leq 5, 1 \leq j \leq 5, i \neq j}{|C_j|}-k\right)=\\
                =\frac{1}{10}\left(\frac{1}{2} \cdot 4 \sum_{1 \leq i \leq 5}{|C_i|}+\frac{1}{2} \cdot 4 \sum_{1 \leq j \leq 5}{|C_j|}-k\right)=\\
                =\frac{1}{10}\left(\frac{1}{2} \cdot 4n+\frac{1}{2} \cdot 4n-k\right)=\frac{1}{10}(4n-k)=\frac{2}{5}n-\frac{k}{10},
            \end{multline*}
            as claimed.
        \end{proof}

        Notice that the proof of Theorem~\ref{theorem:lower_bound_for_multigraphs_2n/5-k/10} uses the fact that the acyclic $5$-coloring is proper (we need it to establish that there are no parallel edges inside any single color class), while the proof of Corollary~\ref{corollary:lower_bound_for_multigraphs_2n/5-3k/5} does not use that fact. So, Theorem~\ref{theorem:lower_bound_for_multigraphs_2n/5-k/10} uses a stronger condition on the coloring than Corollary~\ref{corollary:lower_bound_for_multigraphs_2n/5-3k/5} and results in a stronger lower bound on $a(M)$.

        The lower bounds from Theorem~\ref{theorem:lower_bound_for_multigraphs_2n/5-k/10} and Corollary~\ref{corollary:ab_conjecture_implies_lower_bound_for_multigraphs_n/2-k/2} are better than the general lower bound from Theorem~\ref{theorem:planar_multigraphs_a_geq_n/4_tight} when $k$ is small. More precisely, the lower bound from Theorem~\ref{theorem:lower_bound_for_multigraphs_2n/5-k/10} is better than the general lower bound from Theorem~\ref{theorem:planar_multigraphs_a_geq_n/4_tight} when $k<\frac{3}{2}n$; and the lower bound from Corollary~\ref{corollary:ab_conjecture_implies_lower_bound_for_multigraphs_n/2-k/2} is better than the general lower bound from Theorem~\ref{theorem:planar_multigraphs_a_geq_n/4_tight} when $k<\frac{n}{2}$.
        
        Also, notice that, when comparing the lower bounds from Theorem~\ref{theorem:lower_bound_for_multigraphs_2n/5-k/10} and Corollary~\ref{corollary:ab_conjecture_implies_lower_bound_for_multigraphs_n/2-k/2} with each other, the lower bound from Theorem~\ref{theorem:lower_bound_for_multigraphs_2n/5-k/10} has a smaller coefficient of $n$ but at the same time a smaller absolute value of the negative coefficient of $k$, so, under different relationships between $n$ and $k$, either of these two lower bounds can be better than the other. Specifically, when $k<\frac{n}{4}$, the lower bound from Corollary~\ref{corollary:ab_conjecture_implies_lower_bound_for_multigraphs_n/2-k/2} is better; when $k>\frac{n}{4}$, the lower bound from Theorem~\ref{theorem:lower_bound_for_multigraphs_2n/5-k/10} is better; and, when $k=\frac{n}{4}$, they are equal.

    \section{The variant without $2$-faces}\label{section:without_2-faces}
        \subsection{Preliminaries}
            Since the original problem about $a(M)$ for planar multigraphs is too easily reduced to the problem about the independence number for simple planar graphs, it is natural to seek to modify the problem slightly by introducing an additional restriction that would make the problem less trivial. We consider such a restriction here by prohibiting $2$-faces. Specifically, we consider the problem of bounding $a(M)$ from below in terms of $n$ for planar multigraphs on $n$ vertices that admit an embedding into the plane without $2$-faces. Notice that parallel edges are still permitted, but they are prohibited from forming a $2$-face.
            
            This class of multigraphs can be viewed as an intermediate case between planar multigraphs and simple planar graphs. Furthermore, informally, it can be described as \emph{locally simple} in the sense that, if we look separately at each edge and what is immediately nearby it in the embedding, the graph appears simple and the parallel edges are not visible there. To see them, we need to look at a larger region of the embedding, potentially, globally, at the entire embedding. Also, as we will see below, this class of multigraphs shares some properties with simple planar graphs that do not hold for arbitrary planar multigraphs. In particular, the inequality $m \leq 3n-6$ that we will prove in Lemma~\ref{lemma:planar_multigraph_without_2-faces_m_leq_3n-6} below. So, studying the problem for this intermediate class of multigraphs can be motivated by an attempt to take a step from the too easy case of multigraphs towards the too difficult case of the AB conjecture for simple graphs.

            The property of not having $2$-faces is a property of a specific embedding into the plane, and it can hold for one embedding and not hold for another embedding of the same multigraph. However, the property of admitting an embedding into the plane without $2$-faces (which means that at least one such embedding exists, regardless of whether there are also other embeddings that do have $2$-faces) is a property of the multigraph itself. This is similar to how planarity, which is defined by the existence of an embedding into the plane, is a property of the multigraph itself.

            First, let us clarify what exactly we mean by a \emph{$2$-face} (also known as a \emph{digon}).

            \begin{definition}[{\cite[p.~609]{verdiere_computational_topology_of_graphs_on_surfaces}}]\label{definition:degree_of_face}
                We consider each edge as having two sides, and we consider incidences between a side of an edge and a face. The number of such incidences that a given face $f$ has is said to be the \emph{degree} (also known as the \emph{size}) of that face and is denoted by $\deg(f)$. A face of degree $d$ is also called a \emph{$d$-face}.
            \end{definition}

            In particular, if an edge is incident to the same face on both sides, then this edge contributes $2$ to the number of incidences in the degree of that face.

            A $2$-face has exactly $2$ incidences with sides of edges, and notice that, if the graph contains at least $2$ edges, these could be only from two parallel edges (a degenerate case is when the graph contains a single edge and hence a single face, which has degree $2$). Notice that it cannot contain any other edges in its closure, even dangling inside the face from one of the vertices of the boundary or completely inside the face disconnected from the boundary and from the rest of the graph. The existence of any such edges in the closure of the face between two parallel edges would increase its degree to be larger than $2$. A $2$-face can contain in its closure only two parallel edges and their two endpoint vertices, an arbitrary number of additional isolated vertices, and nothing else.

            \begin{lemma}[{\cite[Theorem~7.5.2, p.~326]{gross_yellen_anderson_book_graph_theory}}]\label{lemma:planar_multigraph_2m_equal_sum_of_deg_of_faces}
                In any plane multigraph $\mathcal{M}$ with $m$ edges, we have $2m=\sum_f{\deg(f)}$, where the sum is taken over all faces $f$ of $\mathcal{M}$.
            \end{lemma}
            \begin{proof}
                As in Definition~\ref{definition:degree_of_face}, we consider each edge as having $2$ sides, and we count incidences between a side of an edge and a face in two ways. On the one hand, each edge contributes to that count exactly $2$ incidences on the two sides of the edge (it could be the same face on both sides, but we still count them as $2$ different incidences because they are on the different sides of the edge). On the other hand, each face $f$ contributes to that count exactly its degree $\deg(f)$. Therefore, the number of incidences, on the one hand, is equal to $2m$, and, on the other hand, is equal to $\sum_f{\deg(f)}$. Hence, we have $2m=\sum_f{\deg(f)}$.
            \end{proof}
            
            The general lower bound $a(M) \geq \frac{n}{4}$ from Theorem~\ref{theorem:planar_multigraphs_a_geq_n/4_tight} gives the lower bound for the variant without $2$-faces as well. But the construction showing the tightness of the general lower bound fails for this variant as it uses $2$-faces (below, we will prove that any embedding into the plane of the multigraph from that construction necessarily has $2$-faces). On the other hand, trivially, planar multigraphs that admit an embedding into the plane without $2$-faces include all simple planar graphs. In particular, we can consider a construction of the disjoint union of multiple copies of $K_4$ (without any parallel edges), which has $a(M)=\frac{n}{2}$ (it is the same construction that shows the tightness of the AB conjecture for simple planar graphs). This shows that the lower bound on $a(M)$ for planar multigraphs that admit an embedding into the plane without $2$-faces cannot exceed $\frac{n}{2}$. So, the best possible lower bound on $a(M)$ for planar multigraphs that admit an embedding into the plane without $2$-faces, trivially, must be somewhere between $\frac{n}{4}$ and $\frac{n}{2}$. Below, we will prove improvements to both ends of this trivial interval.

        \subsection{The weak lower bound}
            \begin{lemma}\label{lemma:planar_multigraph_without_2-faces_m_leq_3n-6}
                In any planar multigraph $M$ on $n \geq 3$ vertices with $m$ edges that admits an embedding into the plane without $2$-faces, we have $m \leq 3n-6$.
            \end{lemma}
            \begin{proof}
                Let $M$ have $p$ connected components. Let the embedding $\mathcal{M}$ of $M$ into the plane have $\ell$ faces. If $m=0$, then the claimed inequality $m \leq 3n-6$ holds because of the assumption $n \geq 3$. So, in what follows, we assume that $m \geq 1$. Then there are no $0$-faces (which are possible only if $M$ is edgeless and then there is a single face in $\mathcal{M}$, which has degree $0$). Also, $1$-faces are not possible at all. And, since $\mathcal{M}$ has no $2$-faces, the degree of each face is at least $3$. By Lemma~\ref{lemma:planar_multigraph_2m_equal_sum_of_deg_of_faces}, we have $2m=\sum_f{\deg(f)}$. Therefore, we have $2m=\sum_f{\deg(f)} \geq \sum_f{3} = 3\ell$. Putting this inequality $2m \geq 3\ell$ into Euler's formula $n-m+\ell=1+p$, which holds for both simple planar graphs and for planar multigraphs~\cite[Theorem~7.5.7, p.~328]{gross_yellen_anderson_book_graph_theory} (in~\cite[Theorem~7.5.7, p.~328]{gross_yellen_anderson_book_graph_theory}, Euler's formula is proved as $n-m+\ell=2$ for connected plane multigraphs, but it can be easily extended to $n-m+\ell=1+p$ for any, not necessarily connected, plane multigraphs using induction on $p$), we get $1+p=n-m+\ell \leq n-m+\frac{2}{3}m$, which simplifies to $m \leq 3n-3-3p$. Since $p \geq 1$, we have $m \leq 3n-3-3p \leq 3n-6$, as claimed.
            \end{proof}

            The same inequality $m \leq 3n-6$ as in Lemma~\ref{lemma:planar_multigraph_without_2-faces_m_leq_3n-6} is well-known to hold for simple planar graphs~\cite[Corollary~4.2.10, p.~102]{diestel_book_graph_theory}. The key fact is that each face has degree at least $3$, which is true both for simple plane graphs and for plane multigraphs without $2$-faces (for $n \geq 3$), but is not true for plane multigraphs with $2$-faces. The existence of $2$-faces allows plane multigraphs to have an unlimited number of edges (unbounded in terms of the number of vertices), which can be seen by an example of a multigraph on $2$ vertices with an arbitrary number of parallel edges between them.

            The construction $M_k$ of a disjoint union of $k$ copies of $K_4$ with all edges duplicated on which the lower bound $a(M) \geq \frac{n}{4}$ from Theorem~\ref{theorem:planar_multigraphs_a_geq_n/4_tight} is attained has $n=4k$ vertices and $m=12k$ edges. So, the inequality $m \leq 3n-6$ from Lemma~\ref{lemma:planar_multigraph_without_2-faces_m_leq_3n-6} can be rewritten as $12k \leq 12k-6$ for that multigraph $M_k$, and we see that it does not hold. This implies that that multigraph $M_k$ does not admit an embedding into the plane without $2$-faces.

            Lemma~\ref{lemma:planar_multigraph_without_2-faces_m_leq_3n-6} in combination with Theorem~\ref{theorem:lower_bound_for_multigraphs_2n/5-k/10} can be used to prove a lower bound on $a(M)$ for the variant without $2$-faces that is better than the general lower bound from Theorem~\ref{theorem:planar_multigraphs_a_geq_n/4_tight}, but only very marginally: by a small additive constant.

            \begin{theorem}\label{theorem:without_2-faces_lower_bound_n/4+3/10}
                For every planar multigraph $M$ on $n \geq 1$ vertices admitting an embedding into the plane without $2$-faces, we have $a(M) \geq \frac{n}{4}+\frac{3}{10}$.
            \end{theorem}
            \begin{proof}
                If $n=1$ or $n=2$, then we can take any single vertex as an induced forest, so $a(M) \geq 1$, and the claimed lower bound holds. In what follows, we assume that $n \geq 3$. Denote by $m$ the number of edges in $M$ and by $k$ the number of pairs of vertices in $M$ with parallel edges between them. Since every pair of vertices with parallel edges between them has at least $2$ edges between them, we have $2k \leq m$. By Lemma~\ref{lemma:planar_multigraph_without_2-faces_m_leq_3n-6}, we have $m \leq 3n-6$. Combining these two inequalities, we get $2k \leq m \leq 3n-6$, which implies that $k \leq \frac{3n-6}{2}$. Now, using Theorem~\ref{theorem:lower_bound_for_multigraphs_2n/5-k/10}, we have $a(M) \geq \frac{2}{5}n-\frac{k}{10} \geq \frac{2}{5}n-\frac{1}{10} \cdot \frac{3n-6}{2}=\frac{n}{4}+\frac{3}{10}$, as claimed.
            \end{proof}

        \subsection{The forest structure of $2$-cycles}\label{subsection:forest_structure_of_2-cycles}
            In this subsection, we prove that a subset of the set of $2$-cycles in a plane multigraph such that any two $2$-cycles from that subset have at most $1$ common vertex, equipped with the relation of one $2$-cycle containing another $2$-cycle immediately inside it, forms a rooted directed forest. The full formal proof of that statement is straightforward, but lengthy. Informally, at a high level, the argument is that the Jordan curve theorem implies that any two $2$-cycles from the subset cannot cross, but only touch at a single common point, which implies that the set of the interiors of the $2$-cycles is laminar (for any two interiors, either they do not intersect or one of them lies entirely inside the other), and the laminarity condition is known to induce a forest-like structure. Those readers who trust such an informal argument can skip the rest of this subsection entirely.
            
            By the Jordan curve theorem, any simple closed Jordan curve $C$ divides the plane into two open regions: the interior and the exterior, which we denote by $\operatorname{Int}(C)$ and $\operatorname{Ext}(C)$, respectively. Denote by $\overline{\operatorname{Int}}(C)$ and $\overline{\operatorname{Ext}}(C)$ their closures. By the Jordan curve theorem, we have $\overline{\operatorname{Int}}(C)=\operatorname{Int}(C) \cup C$ and $\overline{\operatorname{Ext}}(C)=\operatorname{Ext}(C) \cup C$.
            
            Each pair of parallel edges in a plane multigraph forms a $2$-cycle, which is a simple closed Jordan curve. So, all definitions and statements that are formulated for arbitrary simple closed Jordan curves also apply to $2$-cycles in a plane multigraph. For example, for a $2$-cycle $C$, we will use the same notation $\operatorname{Int}(C)$, $\operatorname{Ext}(C)$, $\overline{\operatorname{Int}}(C)$, $\overline{\operatorname{Ext}}(C)$ as for an arbitrary simple closed Jordan curve.
            
            \begin{lemma}\label{lemma:jordan_curves_inside_equivalence}
                Let $C$ and $D$ be two simple closed Jordan curves in the plane. Then $D \subseteq \overline{\operatorname{Int}}(C)$ if and only if $\operatorname{Int}(D) \subseteq \operatorname{Int}(C)$.
            \end{lemma}
            \begin{proof}
                Assume that $D \subseteq \overline{\operatorname{Int}}(C)$. This implies that $D$ does not intersect $\operatorname{Ext}(C)$. Since the set $\operatorname{Ext}(C)$ is connected (by the Jordan curve theorem for $C$) and does not have common points with $D$, by the Jordan curve theorem for $D$, it must lie entirely in one of the two connected components $\operatorname{Int}(D)$ or $\operatorname{Ext}(D)$. But $\operatorname{Ext}(C)$ is an unbounded set and thus it cannot lie in the bounded set $\operatorname{Int}(D)$. The only remaining possibility is that $\operatorname{Ext}(C) \subseteq \operatorname{Ext}(D)$. Taking closures, we get $\overline{\operatorname{Ext}}(C) \subseteq \overline{\operatorname{Ext}}(D)$. Taking complements, we get $\operatorname{Int}(D) = \mathbb{R}^2 \setminus \overline{\operatorname{Ext}}(D) \subseteq \mathbb{R}^2 \setminus \overline{\operatorname{Ext}}(C) = \operatorname{Int}(C)$, as required.

                Conversely, assume that $\operatorname{Int}(D) \subseteq \operatorname{Int}(C)$. Taking closures, we get $\overline{\operatorname{Int}}(D) \subseteq \overline{\operatorname{Int}}(C)$. Therefore, we have $D \subseteq D \cup \operatorname{Int}(D) = \overline{\operatorname{Int}}(D) \subseteq \overline{\operatorname{Int}}(C)$, as required.
            \end{proof}

            \begin{lemma}\label{lemma:touching_jordan_curves}
                For any two simple closed Jordan curves $C$ and $D$ in the plane having at most $1$ common point, we have either $D \subseteq \overline{\operatorname{Int}}(C)$ or $D \subseteq \overline{\operatorname{Ext}}(C)$.
            \end{lemma}
            \begin{proof}
                Slightly informally, the statement of the lemma follows from the observation that $D$ cannot cross $C$ from the interior of $C$ to the exterior of $C$ at the single common point (if it exists) and can only touch $C$, staying on the same side of it (either in the interior or in the exterior), because, by the Jordan curve theorem, to return back to the interior, it would need to cross it again at a second point.
                
                More formally, this argument can be presented as follows. Denote $D'=D \setminus (C \cap D)$. Clearly, $D'$ has no common points with $C$. The set $D'$ is homeomorphic to either a circle, if $C$ and $D$ have no common points, or an open interval, if $C$ and $D$ have a single common point. In both cases, for any two points on $D'$, there is a subarc of $D'$, which is a simple Jordan arc, connecting these two points. If $D'$ contains points on both sides of $C$ (in its interior and in its exterior), then we take two such points on different sides of $C$ and, by the Jordan curve theorem, the subarc of $D'$ connecting these two points must have a common point with $C$, which is a contradiction. Therefore, $D'$ must lie either completely in $\operatorname{Int}(C)$ or completely in $\operatorname{Ext}(C)$. Taking closures, we conclude that $D$, which is the closure of $D'$, must lie either completely in $\overline{\operatorname{Int}}(C)$ or completely in $\overline{\operatorname{Ext}}(C)$.
            \end{proof}

            The distinction between crossing and touching simple closed Jordan curves can be characterized in terms of the local cyclic order of these two curves around their common point. We will formally define and prove it only for $2$-cycles in a plane multigraph. Let $C$ and $D$ be two $2$-cycles in a plane multigraph having exactly $1$ common vertex $v$ (and thus not having common edges). Denote the half-edges of $C$ incident to $v$ by $e_{C1}$ and $e_{C2}$ and the half-edges of $D$ incident to $v$ by $e_{D1}$ and $e_{D2}$. It is easy to see that, in the local cyclic order at $v$ of all the half-edges of the plane multigraph incident to $v$ (in the rotation system of the plane multigraph at $v$), there are two possible local cyclic suborders of the half-edges $e_{C1}$, $e_{C2}$, $e_{D1}$, $e_{D2}$ up to a cyclic shift and up to a renaming of $e_{C1}$ and $e_{C2}$ and of $e_{D1}$ and $e_{D2}$: $(e_{C1},e_{C2},e_{D1},e_{D2})$ and $(e_{C1},e_{D1},e_{C2},e_{D2})$. In the case of the local cyclic suborder $(e_{C1},e_{C2},e_{D1},e_{D2})$, we will say that the \emph{local cyclic order} (or the \emph{rotation system}) of $C$ and $D$ at $v$ is \emph{consecutive} and denote it by the pattern $CCDD$. In the case of the local cyclic suborder $(e_{C1},e_{D1},e_{C2},e_{D2})$, we will say that the \emph{local cyclic order} (or the \emph{rotation system}) of $C$ and $D$ at $v$ is \emph{alternating} and denote it by the pattern $CDCD$.

            \begin{lemma}\label{lemma:2-cycles_with_single_common_point_consecutive_cyclic_order}
                For any two $2$-cycles $C$ and $D$ in a plane multigraph having exactly $1$ common vertex $v$, the local cyclic order of $C$ and $D$ at $v$ is consecutive $CCDD$.
            \end{lemma}
            \begin{proof}
                Suppose, to the contrary, that the local cyclic order of $C$ and $D$ at $v$ is alternating $CDCD$. Since the local cyclic order at $v$ of all the half-edges of the plane multigraph incident to $v$ (the rotation system at $v$) is well-defined, there exists a sufficiently small neighborhood $U$ of $v$, homeomorphic to an open disk, such that its boundary $B$, which is a simple closed Jordan curve and is homeomorphic to a circle, intersects each edge incident to $v$ at exactly $1$ point. The cyclic order of these intersection points on $B$ is precisely what defines the local cyclic order of the half-edges at $v$. Denote by $\overline{U}$ the closure of $U$. The condition that $U$ is sufficiently small implies that we can choose it in such a way that neither $C$ nor $D$ lies entirely in $\overline{U}$. Denote the half-edges of $C$ incident to $v$ by $e_{C1}$ and $e_{C2}$ and the half-edges of $D$ incident to $v$ by $e_{D1}$ and $e_{D2}$. Denote by $c_1$, $c_2$, $d_1$, $d_2$ the intersection points of $B$ with the edges corresponding to the half-edges $e_{C1}$, $e_{C2}$, $e_{D1}$, $e_{D2}$, respectively. By definition, the alternating local cyclic order $CDCD$ means that the local cyclic suborder of the half-edges is $(e_{C1},e_{D1},e_{C2},e_{D2})$ up to a cyclic shift and up to a renaming of $e_{C1}$ and $e_{C2}$ and between $e_{D1}$ and $e_{D2}$. This, in turn, implies that the local cyclic suborder of the points $c_1$, $c_2$, $d_1$, $d_2$ on $B$ is $(c_1,d_1,c_2,d_2)$ up to a cyclic shift and up to a renaming of $c_1$ and $c_2$ and between $d_1$ and $d_2$.

                The simple closed Jordan curve $B$ is divided by the points $c_1$ and $c_2$ into two open arcs that we denote by $B_1$ and $B_2$. Neither $B_1$ nor $B_2$ has common points with $C$. Therefore, the Jordan curve theorem for $C$ implies that each of them lies entirely either in the interior of $C$ or in the exterior of $C$ because, if there were two points on $B_i$ on different sides of $C$ (one in the interior and one in the exterior), then, by the Jordan curve theorem for $C$, the subarc of $B_i$ connecting these two points would have to cross $C$, which would contradict the fact that $B_i$ has no common points with $C$.

                Now, let us prove that $B_1$ and $B_2$ lie on different sides of $C$: one in the interior and one in the exterior. Suppose, to the contrary, that they lie on the same side of $C$. Without loss of generality, assume that they both lie in the exterior of $C$. Then $B$ does not intersect the interior of $C$. By the Jordan curve theorem, $C$ is the boundary of that interior. By the definition of the boundary, every neighborhood of every point of the boundary $C$ intersects both the interior of $C$ and the exterior of $C$. Since $v$ lies on that boundary $C$ and $U$ is a neighborhood of $v$, $U$ must intersect the interior of $C$. Therefore, there exists a point $p \in \operatorname{Int}(C) \cap U$. If there existed a point $x$ in the interior of $C$ and outside of $\overline{U}$, then there would exist a path in the interior of $C$ between $p$ and $x$ (because, by the Jordan curve theorem for $C$, the interior of $C$ is connected and hence path-connected), and then, by the Jordan curve theorem for $B$, that path would have to cross $B$, which would contradict the fact that $B$ does not intersect the interior of $C$. Therefore, the interior of $C$ cannot contain any points outside of $\overline{U}$. In other words, we have $\operatorname{Int}(C) \subseteq \overline{U}$. Taking closures, we obtain that $\overline{\operatorname{Int}}(C) \subseteq \overline{U}$. In particular, $C$, which is the boundary of $\overline{\operatorname{Int}}(C)$, lies entirely inside $\overline{U}$, which contradicts the assumption that $U$ was chosen sufficiently small for $C$ not to lie entirely in $\overline{U}$. This concludes the proof that $B_1$ and $B_2$ lie on different sides of $C$.
                
                It is easy to see that the cyclic order $(c_1,d_1,c_2,d_2)$ implies that $d_1$ and $d_2$ belong to different parts of $B \setminus \{c_1,c_2\}$: either $d_1 \in B_1$ and $d_2 \in B_2$ or $d_1 \in B_2$ and $d_2 \in B_1$. This implies that $D$ has points (specifically, the points $d_1$ and $d_2$) both in $B_1$ and in $B_2$, and hence both in the interior of $C$ and in the exterior of $C$, which contradicts Lemma~\ref{lemma:touching_jordan_curves}.
            \end{proof}

            \begin{lemma}\label{lemma:touching_jordan_curves_interior_in_exterior_subsets_1}
                Let $C$ and $D$ be two simple closed Jordan curves in the plane having at most $1$ common point. Assume that $D \subseteq \overline{\operatorname{Ext}}(C)$. Then either $\operatorname{Int}(C) \subseteq \operatorname{Int}(D)$ or $\operatorname{Int}(C) \subseteq \operatorname{Ext}(D)$.
            \end{lemma}
            \begin{proof}
                Since $D \subseteq \overline{\operatorname{Ext}}(C)$, $D$ does not intersect with $\operatorname{Int}(C)$. Since $\operatorname{Int}(C)$ is a connected set (by the Jordan curve theorem for $C$) that does not intersect $D$, by the Jordan curve theorem for $D$, it must lie entirely in one of the two connected components $\operatorname{Int}(D)$ or $\operatorname{Ext}(D)$, that is, we have either $\operatorname{Int}(C) \subseteq \operatorname{Int}(D)$ or $\operatorname{Int}(C) \subseteq \operatorname{Ext}(D)$, as required.
            \end{proof}

            \begin{lemma}\label{lemma:touching_jordan_curves_interior_in_exterior_subsets_2}
                Let $C$ and $D$ be two simple closed Jordan curves in the plane having at most $1$ common point. Assume that $D \subseteq \overline{\operatorname{Ext}}(C)$. Then either $\operatorname{Int}(D) \subseteq \operatorname{Ext}(C)$ or $\operatorname{Ext}(D) \subseteq \operatorname{Ext}(C)$.
            \end{lemma}
            \begin{proof}
                By Lemma~\ref{lemma:touching_jordan_curves_interior_in_exterior_subsets_1}, we have either $\operatorname{Int}(C) \subseteq \operatorname{Int}(D)$ or $\operatorname{Int}(C) \subseteq \operatorname{Ext}(D)$. Taking closures, we get that either $\overline{\operatorname{Int}}(C) \subseteq \overline{\operatorname{Int}}(D)$ or $\overline{\operatorname{Int}}(C) \subseteq \overline{\operatorname{Ext}}(D)$. Taking complements, we get that either $\operatorname{Ext}(D) = \mathbb{R}^2 \setminus \overline{\operatorname{Int}}(D) \subseteq \mathbb{R}^2 \setminus \overline{\operatorname{Int}}(C) = \operatorname{Ext}(C)$, which is the second option in the statement of the lemma, or $\operatorname{Int}(D) = \mathbb{R}^2 \setminus \overline{\operatorname{Ext}}(D) \subseteq \mathbb{R}^2 \setminus \overline{\operatorname{Int}}(C) = \operatorname{Ext}(C)$, which is the first option in the statement of the lemma.
            \end{proof}

            \begin{lemma}\label{lemma:touching_jordan_curves_interiors_subsets}
                For any two simple closed Jordan curves $C$ and $D$ in the plane having at most $1$ common point, we have either $\operatorname{Int}(D) \subseteq \operatorname{Int}(C)$, or $\operatorname{Int}(C) \subseteq \operatorname{Int}(D)$, or $\operatorname{Int}(C) \cap \operatorname{Int}(D)=\varnothing$.
            \end{lemma}
            \begin{proof}
                By Lemma~\ref{lemma:touching_jordan_curves}, we have either $D \subseteq \overline{\operatorname{Int}}(C)$ or $D \subseteq \overline{\operatorname{Ext}}(C)$. We analyze these two cases separately.

                The first case is when $D \subseteq \overline{\operatorname{Int}}(C)$. By Lemma~\ref{lemma:jordan_curves_inside_equivalence}, we have $\operatorname{Int}(D) \subseteq \operatorname{Int}(C)$, which is the first option in the statement of the lemma.
                
                The second case is when $D \subseteq \overline{\operatorname{Ext}}(C)$. By Lemma~\ref{lemma:touching_jordan_curves_interior_in_exterior_subsets_1}, we have either $\operatorname{Int}(C) \subseteq \operatorname{Int}(D)$, which is the second option in the statement of the lemma, or $\operatorname{Int}(C) \subseteq \operatorname{Ext}(D)$, which implies $\operatorname{Int}(C) \cap \operatorname{Int}(D)=\varnothing$, which is the third option in the statement of the lemma.
            \end{proof}

            \begin{definition}[{\cite[the beginning of the introduction, p.~510]{cjr_on_2-coverings_and_2-packings_of_laminar_families}}]
                A set $S$ of sets is said to be \emph{laminar} if, for any two sets $A,B \in S$, we have either $B \subseteq A$, or $A \subseteq B$, or $A \cap B = \varnothing$.
            \end{definition}

            \begin{corollary}\label{corollary:touching_jordan_curves_interiors_laminar}
                Let $S$ be a set of simple closed Jordan curves in the plane such that any two distinct curves from $S$ have at most $1$ common point. Let $I_S=\{\operatorname{Int}(C) \mid C \in S\}$ be the set of interiors of the curves from $S$. Then the set $I_S$ is laminar.
            \end{corollary}
            \begin{proof}
                Take two arbitrary sets $A,B \in I_S$ that are the interiors of curves $C,D \in S$, respectively, that is, $A=\operatorname{Int}(C)$ and $B=\operatorname{Int}(D)$. If $C=D$, then we have $A=\operatorname{Int}(C)=\operatorname{Int}(D)=B$ and, in particular, $B \subseteq A$. If $C \neq D$, then we can apply Lemma~\ref{lemma:touching_jordan_curves_interiors_subsets} and obtain that either $\operatorname{Int}(D) \subseteq \operatorname{Int}(C)$, or $\operatorname{Int}(C) \subseteq \operatorname{Int}(D)$, or $\operatorname{Int}(C) \cap \operatorname{Int}(D)=\varnothing$, that is, either $B \subseteq A$, or $A \subseteq B$, or $A \cap B = \varnothing$. Thus, the set $I_S$ satisfies the definition of a laminar set.
            \end{proof}

            For two simple closed Jordan curves (and, in particular, for two $2$-cycles in a plane multigraph) $C$ and $D$, we say that $D$ is \emph{inside} $C$ or that $C$ contains $D$ inside and write $D \preceq C$ or $C \succeq D$ if one of the two equivalent conditions from Lemma~\ref{lemma:jordan_curves_inside_equivalence} holds, that is, if $D \subseteq \overline{\operatorname{Int}}(C)$ or, equivalently, if $\operatorname{Int}(D) \subseteq \operatorname{Int}(C)$.

            Counterintuitively, for partially ordered sets of simple closed Jordan curves, instead of the binary relation $\preceq$ of being inside, we will need to use the opposite binary relation $\succeq$ of containing inside. While both binary relations result in partially ordered sets, only the binary relation $\succeq$, and not $\preceq$, under certain conditions, results in a forest poset (the formal definition will be given below). So, the usage of $\succeq$ instead of $\preceq$ in the lemmas below, while it might initially look counterintuitive, is deliberate.

            \begin{lemma}\label{lemma:jordan_curves_contain_inside_isomorphism_to_supseteq_of_interiors}
                Let $S$ be a set of simple closed Jordan curves in the plane. Let $I_S=\{\operatorname{Int}(C) \mid C \in S\}$ be the set of interiors of the curves from $S$. Then the relational structures $(S,\succeq)$ and $(I_S,\supseteq)$ are isomorphic.
            \end{lemma}
            \begin{proof}
                We observe that $\operatorname{Int} : S \to I_S$ is a mapping between the sets $S$ and $I_S$. By definition, we have $C \succeq D$ for $C,D \in S$ if and only if $\operatorname{Int}(C) \supseteq \operatorname{Int}(D)$. It remains only to prove that the mapping is a bijection between $S$ and $I_S$. The mapping is obviously surjective. Let us prove that it is injective. Let $C,D \in S$ be such that $\operatorname{Int}(C) = \operatorname{Int}(D)$. By the Jordan curve theorem, the boundary of $\operatorname{Int}(C)$ is $C$ and the boundary of $\operatorname{Int}(D)$ is $D$. The equality of the interiors implies the equality of their boundaries. Therefore, we have $C=D$. Thus, the mapping is injective, and the relational structures are isomorphic.
            \end{proof}

            \begin{corollary}\label{corollary:jordan_curves_poset}
                Let $S$ be a set of simple closed Jordan curves in the plane. Then $(S,\succeq)$ is a partially ordered set.
            \end{corollary}
            \begin{proof}
                By Lemma~\ref{lemma:jordan_curves_contain_inside_isomorphism_to_supseteq_of_interiors}, the relational structure $(S,\succeq)$ is isomorphic to $(I_S,\supseteq)$. It is well-known that set inclusion is a partial ordering. Thus, $(I_S,\supseteq)$ is a partially ordered set. Therefore, $(S,\succeq)$ is a partially ordered set as well.
            \end{proof}
            
            We will need some notions and results related to partially ordered sets.
            
            Recall that a partially ordered set $(X,\preceq)$ is called \emph{totally ordered} (or \emph{linearly ordered}) if every pair of elements is comparable, that is, for every $x,y \in X$, we have either $x \preceq y$ or $y \preceq x$. A totally ordered subset of a partially ordered set is also called a \emph{chain}. A partially ordered set is called \emph{well-ordered} if every non-empty subset of it contains a least element. It is well-known that a well-ordered set is always totally ordered~\cite[Exercise~36, p.~25]{moerdijk_oosten_sets_models_proofs} and every \emph{finite} totally ordered set is well-ordered~\cite[Example~(4), p.~25]{moerdijk_oosten_sets_models_proofs}. So, for finite sets, which are our main concern here, these two notions are equivalent.
            
            \begin{definition}[{\cite[Chapter~II, Section~5, Definition~5.1, p.~68]{kunen_set_theory}}]
                A partially ordered set $(X,\preceq)$ is said to be a \emph{forest poset} (also known as a \emph{tree-ordered set} and a \emph{tree}) if, for every element $x \in X$, the set $\{y \in X \mid y \prec x\}$ of predecessors of $x$ is well-ordered. If $X$ is finite, then the condition for the set of predecessors to be well-ordered is equivalent to the condition of it being a chain.
            \end{definition}

            \begin{lemma}\label{lemma:laminar_set_forest_poset}
                Let $S$ be a finite laminar set of non-empty sets. Then $(S,\supseteq)$ is a forest poset.
            \end{lemma}
            \begin{proof}
                Since $S$ is finite, we need to prove that, for every $A \in S$, the set $\{X \in S \mid X \supsetneq A\}$ of its predecessors is a chain. Take two arbitrary sets $B,C \in \{X \in S \mid X \supsetneq A\}$. We have $B \supsetneq A$ and $C \supsetneq A$. This implies that $B \cap C \supseteq A$. By the assumption of non-emptiness of all sets in $S$, the set $A$ is non-empty. Therefore, the sets $B$ and $C$ intersect. Therefore, by the laminarity of $S$, we have either $B \subseteq C$ or $C \subseteq B$. This implies that $B$ and $C$ are comparable under $\supseteq$. Thus, the set $\{X \in S \mid X \supsetneq A\}$ is a chain and $(S,\supseteq)$ is a forest poset.
            \end{proof}

            Notice that, in Lemma~\ref{lemma:laminar_set_forest_poset}, the elements of $S$ (which are sets themselves) are not required to be finite. Only $S$ is required to be finite, while the elements of $S$ are only required to be non-empty and could be finite or infinite. Also, in Lemma~\ref{lemma:laminar_set_forest_poset}, it is important that we use the binary relation $\supseteq$ and not $\subseteq$. The statement does not hold for $\subseteq$. It is for that reason that we need to use the binary relation $\succeq$ instead of $\preceq$ for simple closed Jordan curves.

            Although we were not able to find the exact statement of Lemma~\ref{lemma:laminar_set_forest_poset} in the existing published literature, the general fact that the laminarity condition induces a forest-like structure is known~\cite[Subsection~1.1, p.~512]{cjr_on_2-coverings_and_2-packings_of_laminar_families}, \cite[the discussion immediately after the statement of Lemma~4]{melkonian_tardos_directed_network_design_problem}, \cite[Exercise~114(c)(i), pp.~83--84]{moerdijk_oosten_sets_models_proofs}.

            \begin{corollary}\label{corollary:jordan_curves_forest_poset}
                Let $S$ be a finite set of simple closed Jordan curves in the plane such that any two distinct curves from $S$ have at most $1$ common point. Then $(S,\succeq)$ is a forest poset.
            \end{corollary}
            \begin{proof}
                By Lemma~\ref{lemma:jordan_curves_contain_inside_isomorphism_to_supseteq_of_interiors}, the relational structure $(S,\succeq)$ is isomorphic to $(I_S,\supseteq)$. Since $S$ is finite, $I_S$ is finite as well. The Jordan curve theorem implies that the interior of any simple closed Jordan curve is non-empty (because it is one of the two connected components into which the curve divides the plane, and a connected component is, by definition, non-empty). Therefore, the empty set does not belong to $I_S$. By Corollary~\ref{corollary:touching_jordan_curves_interiors_laminar}, the set $I_S$ is laminar. Thus, all the assumptions of Lemma~\ref{lemma:laminar_set_forest_poset} for $I_S$ hold, and, applying it to $I_S$, we obtain that $(I_S,\supseteq)$ is a forest poset. Finally, the isomorphism implies that $(S,\succeq)$ is a forest poset as well.
            \end{proof}

            \begin{definition}[{\cite[Subsection~2.1, p.~46]{matousek_nesetril_invitation_to_discrete_mathematics}, \cite[Chapter~1, Definition~1.14, p.~11]{davey_priestley_introduction_to_lattices_and_order}}]
                For a partially ordered set $(X,\preceq)$ and elements $x,y \in X$, we say that $x$ \emph{immediately precedes} $y$, or that $y$ \emph{immediately succeeds} $x$, or that $x$ is an \emph{immediate predecessor} of $y$, or that $y$ is an \emph{immediate successor} of $x$ (also known as $x$ \emph{being covered by} $y$ or $y$ \emph{covering} $x$), if $x \prec y$ and there is no element $z \in X$ such that $x \prec z \prec y$. Thus, every partially ordered set $(X,\preceq)$ induces another binary relation on $X$ that is called the \emph{immediate precedence relation} (also known as the \emph{covering relation}).
            \end{definition}

            The immediate precedence relation defines a directed graph called the \emph{immediate precedence directed graph} (also known as the \emph{covering directed graph}) with $X$ as the set of vertices and an arc from $x$ to $y$ if and only if $x$ immediately precedes $y$. The immediate precedence directed graph is directly related to the so-called \emph{Hasse diagram}~\cite[Chapter~1, Definition~1.15, p.~11]{davey_priestley_introduction_to_lattices_and_order} of a partially ordered set, which is effectively a specific way to draw the immediate precedence directed graph (it is usually drawn undirected, but the orientation of the edges is implied by the relative vertical positions of the vertices).

            We say that a directed graph is a \emph{rooted directed forest} if the undirected graph obtained from it by removing the orientations of the edges is a forest in which each connected component has a designated vertex, called the \emph{root}, such that all arcs corresponding to the edges of that connected component are oriented away from that root.

            \begin{lemma}\label{lemma:forest_poset_induces_rooted_directed_forest}
                Let $(X,\preceq)$ be a finite forest poset. Then its immediate precedence directed graph is a rooted directed forest.
            \end{lemma}
            \begin{proof}
                Denote the immediate precedence directed graph of $(X,\preceq)$ by $G$. Denote by $H$ the undirected graph obtained from $G$ by removing the orientations of the edges. To prove that $G$ is a rooted directed forest, by definition, we need to prove that $H$ is a forest and that we can designate a root in each connected component of $H$ such that all arcs of $G$ corresponding to the edges of that connected component are oriented away from that root.

                First, we prove that the in-degree of every vertex in $G$ is at most $1$. Suppose, to the contrary, that there exists a vertex $x \in X$ with the in-degree in $G$ at least $2$. This means that it has two distinct incoming arcs $yx$ and $zx$, where $y \neq z$. This, in turn, means that both $y$ and $z$ immediately precede $x$. In particular, they both precede $x$ and thus belong to the set of predecessors of $x$, which, by the definition of a forest poset, must be a chain. Therefore, $y$ and $z$ must be comparable. Without loss of generality, assume that $y \prec z$. Then, we have $y \prec z \prec x$, which contradicts the assumption that $y$ immediately precedes $x$, as $z$ is an element between them. This concludes the proof that the in-degree of every vertex in $G$ is at most $1$.

                Consider a connected component $C$ of $H$ with the set of vertices $V_C \subseteq X$. By the definition of a connected component, $V_C$ is non-empty. It is well-known that every finite non-empty partially ordered set contains a minimal element~\cite[Theorem~2.2.3, p.~49]{matousek_nesetril_invitation_to_discrete_mathematics}. Therefore, the partially ordered set $(V_C,\preceq)$ must have a minimal element, which we denote by $r$. We claim that $r$ has no incoming arcs in $G$, that is, its in-degree is $0$. Indeed, if, to the contrary, there existed an incoming arc $xr$, then we would have $x \prec r$ and $x$ would belong to the same connected component $C$ of $H$, which would contradict the minimality of $r$.

                Since the in-degree of $r$ is $0$ and the in-degree of every other vertex in $V_C$ is at most $1$, we can bound from above the sum of the in-degrees of the vertices in $V_C$ by $0 + (|V_C|-1) \cdot 1 = |V_C|-1$. On the other hand, every arc corresponding to an edge in the connected component $C$ contributes exactly $1$ to that sum of the in-degrees of the vertices in $V_C$. Hence, that sum is exactly equal to the number of edges in $C$. Therefore, the number of edges in $C$ is at most $|V_C|-1$, that is, at most the number of vertices in $C$ minus $1$. It is well-known that a connected graph with such a property is a tree and the number of edges in it is exactly equal to the number of vertices minus $1$~\cite[Subsection~I.6, Exercise~6, p.~28]{bollobas_modern_graph_theory}. Therefore, $C$ is a tree and the number of edges in it is exactly equal to $|V_C|-1$. This implies that the upper bound on the sum of the in-degrees of the vertices in $V_C$ established above is attained, which can happen only when $r$ is the only vertex in $V_C$ with in-degree $0$ and the in-degrees of all other vertices in $V_C$ are exactly $1$.
                
                Since $C$ is a tree and this holds for every connected component $C$ of $H$, $H$ is a forest.

                We designate $r$ as the root of the connected component $C$. Let us prove that all arcs corresponding to the edges of $C$ are oriented away from $r$. Since $C$ is connected, every edge of $C$ lies on some path starting from $r$. Consider an arbitrary path $rv_1 \ldots v_q$ in $C$ starting from $r$. Since $r$ has no incoming arcs, the arc corresponding to the edge $rv_1$ must be oriented from $r$ to $v_1$. Then $v_1$ has the incoming arc $rv_1$. Since its in-degree is exactly $1$, it has no other incoming arcs. Hence, the arc corresponding to the edge $v_1v_2$ must be oriented from $v_1$ to $v_2$. Repeating the same argument sequentially for all subsequent edges on the path, we conclude that the arcs corresponding to all edges $v_iv_{i+1}$ on the path are oriented from $v_i$ to $v_{i+1}$, that is, away from $r$, as required.
            \end{proof}

            Let $S$ be a set of simple closed Jordan curves in the plane. For two distinct simple closed Jordan curves (and, in particular, for two distinct $2$-cycles in a plane multigraph) $C,D \in S$, we say that $D$ is \emph{immediately inside} $C$ or that $C$ contains $D$ immediately inside it if $D$ is inside $C$ and there is no other simple closed Jordan curve $E \in S$ such that $D$ is inside $E$ and $E$ is inside $C$. In other words, the relation of being immediately inside is exactly the immediate precedence relation (as defined above for an arbitrary partially ordered set) for the relation $\preceq$ of being inside and the relation of containing immediately inside is exactly the immediate precedence relation for the relation $\succeq$ of containing inside. Thus, for a set $S$ of simple closed Jordan curves in the plane (and, in particular, for a subset $S$ of the set of $2$-cycles in a plane multigraph), the immediate precedence directed graph of the partially ordered set $(S,\succeq)$ is the directed graph with $S$ as its set of vertices and an arc from a curve $C$ to a curve $D$ if and only if $D$ is immediately inside $C$.

            \begin{lemma}\label{lemma:jordan_curves_precedence_graph_rooted_directed_forest}
                Let $S$ be a finite set of simple closed Jordan curves in the plane such that any two distinct curves from $S$ have at most $1$ common point. Then the immediate precedence directed graph of the partially ordered set $(S,\succeq)$ is a rooted directed forest.
            \end{lemma}
            \begin{proof}
                The statement immediately follows from combining Corollary~\ref{corollary:jordan_curves_forest_poset} and Lemma~\ref{lemma:forest_poset_induces_rooted_directed_forest}.
            \end{proof}

            \begin{corollary}\label{corollary:2-cycles_precedence_graph_rooted_directed_forest}
                Let $S$ be a subset of the set of $2$-cycles in a plane multigraph such that any two $2$-cycles from $S$ have at most $1$ common vertex. Then the immediate precedence directed graph of the partially ordered set $(S,\succeq)$ is a rooted directed forest.
            \end{corollary}
            \begin{proof}
                The statement immediately follows from Lemma~\ref{lemma:jordan_curves_precedence_graph_rooted_directed_forest} and the fact that $2$-cycles in a plane multigraph are simple closed Jordan curves.
            \end{proof}

        \subsection{The main lower bound}
            We prove the lower bound $a(M) \geq rn+c$ for every planar multigraph $M$ that admits an embedding into the plane without $2$-faces, where $r>0$ and $c$ are constant (not depending on $n$) parameters. The values of the parameters $r$ and $c$ will be specified in the final part of the proof, but, for now, we keep them as parameters.
                    
            The proof is by contradiction. Suppose that the statement is not true. We consider an edge-minimal counterexample, that is, a counterexample with the smallest number of edges (if there are multiple such edge-minimal counterexamples, then we take any of them arbitrarily).
                    
            Notice that taking an edge-\emph{minimal} counterexample is the opposite of the standard technique of taking a vertex-minimal counterexample and then an edge-\emph{maximal} counterexample among all vertex-minimal ones, which, for simple planar graphs, would be a triangulation. We do this on purpose. First, we do not need even the vertex-minimality of a counterexample, although we could assume it. Second, we need specifically an edge-\emph{minimal} counterexample, not an edge-\emph{maximal} one, to eliminate the possibility of parallel edges with multiplicity greater than $2$ in Lemma~\ref{lemma:without_2-faces_min_counterexample_parallel_edges_multiplicity_geq_3_reducible} below. That lemma is the only place where we use the edge-minimality.
                
            In fact, there is an alternative way of dealing with parallel edges with multiplicity greater than $2$ without assuming the edge-minimality and without Lemma~\ref{lemma:without_2-faces_min_counterexample_parallel_edges_multiplicity_geq_3_reducible}. We will elaborate on that alternative way below.

            Notice that the property of not having $2$-faces is not hereditary. So, a subgraph of a plane multigraph without $2$-faces may or may not have $2$-faces. In particular, removing an arbitrary set of vertices from a plane multigraph without $2$-faces might result in a plane multigraph with $2$-faces. Also, contracting a connected subgraph into a single vertex might also result in a plane multigraph with $2$-faces. This makes standard reduction arguments problematic.
                    
            \begin{lemma}\label{lemma:without_2-faces_min_counterexample_parallel_edges_multiplicity_geq_3_reducible}
                For any $r$ and $c$, all parallel edges in an edge-minimal counterexample $M$ have multiplicity exactly $2$.
            \end{lemma}
            \begin{proof}
                Suppose, to the contrary, that $M$ contains two vertices $u$ and $v$ with $b$ parallel edges between $u$ and $v$, where $b \geq 3$. Denote these parallel edges by $e_1$, \ldots, $e_b$ in the clockwise order that they go out of $u$ and come into $v$ in an embedding $\mathcal{M}$ of $M$ into the plane without $2$-faces. Then these edges divide the plane into the open regions $R_1$, \ldots, $R_b$ such that $R_i$ is the region between $e_i$ and $e_{i+1}$ moving clockwise, where $e_{b+1}=e_1$ (circular indexing). Since $\mathcal{M}$ does not contain $2$-faces, none of the regions $R_i$ are edge-free.
                            
                We remove the edge $e_2$ and denote the resulting plane multigraph by $\mathcal{N}$ and the underlying planar multigraph by $N$. This edge removal keeps the number of vertices, which we denote by $n$, decreases the number of edges by $1$, merges the regions $R_1$ and $R_2$ into a single region, and keeps the regions $R_3$, \ldots, $R_b$ unchanged.
                            
                Clearly, this edge removal does not create new $2$-cycles. Since any $2$-face is a $2$-cycle, a $2$-face could only appear in $\mathcal{N}$ if the interior or the exterior of a preexisting $2$-cycle became edge-free. First, consider the case of a $2$-cycle formed by the remaining parallel edges between $u$ and $v$. Both the interior and the exterior of such a $2$-cycle are not edge-free because each of them contains at least one of the regions $R_i$. The only remaining case is a $2$-cycle $D$ containing at most $1$ of the vertices $u$ and $v$. We claim that then either the interior of $D$ or the exterior of $D$ lies in one of the regions $R_i$. This is geometrically intuitive, but, formally, it can be proved using the lemmas from Subsection~\ref{subsection:forest_structure_of_2-cycles}. Indeed, consider an interior point on one of the edges of $D$. Since edges do not cross, that point cannot lie on any of the edges $e_1$, \ldots, $e_b$, and hence it must lie in one of the regions $R_i$. Denote by $C_i$ the $2$-cycle formed by the edges $e_i$ and $e_{i+1}$. It is easy to see that $R_i$ is either the interior or the exterior of $C_i$. Since $D$ has at most $1$ common vertex with $C_i$, by Lemma~\ref{lemma:touching_jordan_curves}, $D$ lies either in the closure of the interior of $C_i$ or in the closure of the exterior of $C_i$, which means that $D$ lies either in the closure of $R_i$ or in $\mathbb{R}^2 \setminus R_i$. But, since one of its points already lies in $R_i$, it cannot lie in $\mathbb{R}^2 \setminus R_i$. Therefore, it lies in the closure of $R_i$. If $R_i$ is the interior of $C_i$, then, by Lemma~\ref{lemma:jordan_curves_inside_equivalence}, the interior of $D$ lies in $R_i$. If $R_i$ is the exterior of $C_i$, then, by Lemma~\ref{lemma:touching_jordan_curves_interior_in_exterior_subsets_2}, either the interior of $D$ lies in $R_i$ or the exterior of $D$ lies in $R_i$. In both subcases, either the interior of $D$ or the exterior of $D$ lies in $R_i$. The one of these two regions that lies in $R_i$ remains unchanged, so it cannot be edge-free since $\mathcal{M}$ did not contain $2$-faces. And the other of these two regions contains all edges in all the other regions $R_j$, $j \neq i$, so it also cannot be edge-free. So, in all cases, the interior and the exterior of a preexisting $2$-cycle cannot become edge-free, and hence $\mathcal{N}$ does not contain $2$-faces.
                
                \begin{figure}[htb]
                    \centering
                    \setlength{\unitlength}{0.8mm}
                    \begin{picture}(93,94)(27,33)
                        \put(40,80){\circle*{1.5}}
                        \put(100,80){\circle*{1.5}}
                        \put(36,80){\makebox(0,0){$u$}}
                        \put(104,80){\makebox(0,0){$v$}}
                        
                        \thicklines
                        
                        \qbezier(40,80)(70,140)(100,80)
                        \put(44,94){\makebox(0,0){$e_1$}}
                
                        \qbezier(40,80)(70,110)(100,80)
                        \put(42.5,82.4){\color{white}{\circle*{2}}}
                        \put(45.8,85.3){\color{white}{\circle*{2}}}
                        \put(49.4,87.9){\color{white}{\circle*{2}}}
                        \put(53.1,90.3){\color{white}{\circle*{2}}}
                        \put(57.1,92.2){\color{white}{\circle*{2}}}
                        \put(61.3,93.7){\color{white}{\circle*{2}}}
                        \put(65.6,94.7){\color{white}{\circle*{2}}}
                        \put(70.0,95.0){\color{white}{\circle*{2}}}
                        \put(74.4,94.7){\color{white}{\circle*{2}}}
                        \put(78.7,93.7){\color{white}{\circle*{2}}}
                        \put(82.9,92.2){\color{white}{\circle*{2}}}
                        \put(86.9,90.3){\color{white}{\circle*{2}}}
                        \put(90.6,87.9){\color{white}{\circle*{2}}}
                        \put(94.2,85.3){\color{white}{\circle*{2}}}
                        \put(97.5,82.4){\color{white}{\circle*{2}}}
                        \put(52,92){\makebox(0,0){$e_2$}}
                        
                        \qbezier(40,80)(70,80)(100,80)
                        \put(52,82){\makebox(0,0){$e_3$}}
                        
                        \qbezier(40,80)(70,50)(100,80)
                        
                        \qbezier(40,80)(70,20)(100,80)
                        \put(44,66){\makebox(0,0){$e_b$}}

                        \thinlines
                        
                        \put(63,101){\makebox(0,0){$R_1$}}
                        \put(63,87){\makebox(0,0){$R_2$}}
                        \put(63,73){\makebox(0,0){$R_3$}}
                        \put(63,45){\makebox(0,0){$R_b$}}
                        
                        \put(75,99){\circle*{1}}
                        \put(85,99){\circle*{1}}
                        \qbezier(75,99)(80,105)(85,99)
                        \qbezier(75,99)(80,93)(85,99)
                        \put(81.5,98.2){\circle*{1}}
                        \put(78.5,99.8){\circle*{1}}
                        \Line(81.5,98.2)(78.5,99.8)
                        
                        \put(75,86){\circle*{1}}
                        \put(85,86){\circle*{1}}
                        \qbezier(75,86)(80,92)(85,86)
                        \qbezier(75,86)(80,80)(85,86)
                        \put(81.5,85.2){\circle*{1}}
                        \put(78.5,86.8){\circle*{1}}
                        \Line(81.5,85.2)(78.5,86.8)
                        
                        \put(75,74){\circle*{1}}
                        \put(85,74){\circle*{1}}
                        \qbezier(75,74)(80,80)(85,74)
                        \qbezier(75,74)(80,68)(85,74)
                        \put(81.5,73.2){\circle*{1}}
                        \put(78.5,74.8){\circle*{1}}
                        \Line(81.5,73.2)(78.5,74.8)
                        
                        \put(75,61){\circle*{1}}
                        \put(85,61){\circle*{1}}
                        \qbezier(75,61)(80,67)(85,61)
                        \qbezier(75,61)(80,55)(85,61)
                        \put(81.5,60.2){\circle*{1}}
                        \put(78.5,61.8){\circle*{1}}
                        \Line(81.5,60.2)(78.5,61.8)
                        
                        \put(110,45){\circle*{1}}
                        \put(120,45){\circle*{1}}
                        \qbezier(110,45)(115,52)(120,45)
                        \qbezier(110,45)(115,38)(120,45)
                        \put(116.5,44.2){\circle*{1}}
                        \put(113.5,45.8){\circle*{1}}
                        \Line(116.5,44.2)(113.5,45.8)

                        \put(30,40){\circle*{1}}
                        \put(110,120){\circle*{1}}
                        \qbezier(30,40)(15,150)(110,120)
                        \qbezier(30,40)(125,10)(110,120)
                    \end{picture}
                    \caption{The edge removal in the proof of Lemma~\ref{lemma:without_2-faces_min_counterexample_parallel_edges_multiplicity_geq_3_reducible}, illustrated for $b=5$, with various possible arrangements of preexisting $2$-cycles.}
                    \label{picture:without_2-faces_min_counterexample_parallel_edges_multiplicity_geq_3_reducible_preexisting_2-cycle}
                \end{figure}

                Any (simple) cycle in $M$ either is a $2$-cycle between $u$ and $v$ or contains at most $1$ edge from the parallel edges $e_i$ between $u$ and $v$. In both cases, it can be rerouted to avoid the removed edge $e_2$ by replacing $e_2$ with one of the other edges $e_i$, $i \neq 2$, while keeping the same vertices of the cycle. Therefore, a set of vertices induces a forest in $M$ if and only if it induces a forest in $N$. Therefore, we have $a(N)=a(M)$.

                By the edge-minimality of the counterexample $M$, we have that $N$ is not a counterexample, that is, $a(N) \geq rn+c$. Combining this with the equality $a(N)=a(M)$, we get that $a(M)=a(N) \geq rn+c$, a contradiction with the assumption that $M$ is a counterexample.
            \end{proof}

            \begin{lemma}\label{lemma:rooted_directed_forest_weight_lower_bound}
                Let $F$ be a rooted directed forest on $n \geq 1$ vertices. Assume that each vertex of $F$ is assigned a weight such that each leaf has weight $2$, each unary vertex (a vertex of out-degree exactly $1$) has weight $1$, and each branching vertex (a vertex of out-degree at least $2$) has weight $0$ or $1$. Then the total weight of $F$ (the sum of the weights of all vertices of $F$) is at least $n+1$. That lower bound is tight.
            \end{lemma}
            \begin{proof}
                Denote by $\ell$ the number of leaves in $F$. Denote by $u$ the number of unary vertices in $F$. Denote by $b$ the number of branching vertices. Denote by $t$ the number of connected components in $F$. Obviously, we have $\ell+u+b=n$ and $t \geq 1$.
                    
                We can count the number of arcs in $F$ in two different ways: as the sum of in-degrees of all vertices and as the sum of out-degrees of all vertices. So, these two sums must be equal to each other. Since every vertex of $F$ except the roots of the connected components has exactly $1$ incoming arc, the sum of in-degrees is equal to $n-t$. On the other hand, the sum of out-degrees is at least $u+2b$ because leaves have no outgoing arcs, every unary vertex has exactly $1$ outgoing arc, and every branching vertex has at least $2$ outgoing arcs. So, we have $n-t \geq u+2b$.

                The total weight of $F$ can be calculated as at least $\ell \cdot 2+u \cdot 1+b \cdot 0=2\ell+u=2(n-u-b)+u=2n-(u+2b) \geq 2n-(n-t)=n+t \geq n+1$.
                    
                To show the tightness of the lower bound, we consider the construction of $F$ consisting of a single directed path on $n$ vertices. It has a single leaf and $n-1$ unary vertices. So, the total weight is equal to $1 \cdot 2+(n-1) \cdot 1=n+1$, which means that the lower bound is attained on that construction.
            \end{proof}

            \begin{theorem}\label{theorem:without_2-faces_lower_bound_3n/10+7/30}
                For every planar multigraph $M$ on $n \geq 1$ vertices admitting an embedding into the plane without $2$-faces, we have $a(M) \geq \frac{3}{10}n+\frac{7}{30}$.
            \end{theorem}
            \begin{proof}
                Suppose that the statement is not true and let $M$ be an edge-minimal counterexample for $r=\frac{3}{10}$ and $c=\frac{7}{30}$. Fix an embedding $\mathcal{M}$ of $M$ into the plane without $2$-faces.
                
                If $n=1$ or $n=2$, then we can take any single vertex as an induced forest, so $a(M) \geq 1$, and the claimed lower bound holds. In what follows, we assume that $n \geq 3$.

                Denote by $m$ the number of edges in $M$. If $m=0$, then $M$ is edgeless and hence a forest, so we have $a(M) \geq n$, and the claimed lower bound holds. So, in what follows, we assume that $m \geq 1$. Then there are no $0$-faces (which are possible only if $M$ is edgeless and then there is a single face in $\mathcal{M}$, which has degree $0$). Also, $1$-faces are not possible at all. And, since $\mathcal{M}$ has no $2$-faces, the degree of each face is at least $3$.

                We call an edge \emph{parallel} if there is another edge parallel to it. We call an edge \emph{non-parallel} if there are no other edges parallel to it.

                Denote by $\ell$ the number of faces in $\mathcal{M}$. Denote by $p$ the number of connected components in $M$. Denote by $k$ the number of pairs of vertices in $M$ with parallel edges between them. By Lemma~\ref{lemma:without_2-faces_min_counterexample_parallel_edges_multiplicity_geq_3_reducible}, for each of the $k$ pairs of vertices with parallel edges between them, there are exactly $2$ parallel edges between them and they form a single $2$-cycle. In total, the number of $2$-cycles in $M$ is exactly $k$. Therefore, the number of parallel edges is exactly $2k$ and the number of non-parallel edges is exactly $m-2k$.

                By Theorem~\ref{theorem:lower_bound_for_multigraphs_2n/5-k/10}, we have $a(M) \geq \frac{2}{5}n-\frac{k}{10}$. Combining this with the inequality $a(M)<\frac{3}{10}n+\frac{7}{30}$ (because $M$ is a counterexample), we get $\frac{3}{10}n+\frac{7}{30}>a(M) \geq \frac{2}{5}n-\frac{k}{10}$, which simplifies to $3n-3k-7<0$.
                    
                In particular, given that $n \geq 3$, that inequality implies that $k>n-\frac{7}{3} \geq 3-\frac{7}{3}=\frac{2}{3}$. In particular, we have $k \neq 0$.

                If an open region $R$ of the plane contains an edge $e$ without its endpoints and the closure of $R$ contains the endpoints of $e$ (so, one or both of the endpoints of $e$ are allowed to be on the boundary of $R$ rather than in $R$ itself), then, for simplicity, we will say that $R$ \emph{contains} $e$, without explicitly specifying every time that $R$ actually contains $e$ only, possibly, without its endpoints.

                The condition that there are no $2$-faces implies that, for each $2$-cycle $C$, both $\operatorname{Ext}(C)$ and $\operatorname{Int}(C)$ contain other edges of the multigraph.

                For a $2$-cycle $C$, define the \emph{exclusive interior} of $C$ as $\operatorname{Int}(C) \setminus \left(\cup_{D, D \neq C}{\overline{\operatorname{Int}}(D)}\right)$, where the union is over all $2$-cycles $D$ distinct from $C$. It is easy to see that the exclusive interiors of different $2$-cycles do not intersect with each other.

                Recall that, for two distinct $2$-cycles $C$ and $D$, we say that $D$ is \emph{immediately inside} $C$ if $D$ lies in $\overline{\operatorname{Int}}(C)$ and there is no other $2$-cycle $C'$ such that $C'$ lies in $\overline{\operatorname{Int}}(C)$ and $D$ lies in $\overline{\operatorname{Int}}(C')$. If, in addition, $D$ is the only $2$-cycle that is immediately inside $C$, then we say that $D$ is \emph{exclusively immediately inside} $C$. We say that a $2$-cycle that contains other $2$-cycles immediately inside it is \emph{branching} if it contains at least $2$ other $2$-cycles immediately inside it and \emph{non-branching} if it contains exactly $1$ other $2$-cycle immediately inside it. We say that a $2$-cycle is a leaf if it does not contain any other $2$-cycles immediately inside it.

                We call a $2$-cycle \emph{exclusively non-empty} if its exclusive interior contains edges and \emph{exclusively empty} otherwise. Clearly, a $2$-cycle cannot be exclusively empty and a leaf simultaneously.
                            
                It follows from Lemma~\ref{lemma:without_2-faces_min_counterexample_parallel_edges_multiplicity_geq_3_reducible} that the $2$-cycles do not share edges. Therefore, any two $2$-cycles have at most $1$ common vertex. By Corollary~\ref{corollary:2-cycles_precedence_graph_rooted_directed_forest}, the immediate precedence directed graph of the partially ordered set $(S,\succeq)$, where $S$ is the set of all the $2$-cycles in $\mathcal{M}$, (that is, the directed graph with the $2$-cycles as its set of vertices and an arc from a $2$-cycle $C$ to a $2$-cycle $D$ if and only if $D$ is immediately inside $C$) is a rooted directed forest that we denote by $F$.

                It is easy to see that a branching $2$-cycle, as defined above, is a branching vertex in $F$ (its out-degree is at least $2$), a non-branching $2$-cycle is a unary vertex in $F$ (its out-degree is exactly $1$), and a leaf $2$-cycle is a leaf in $F$ (its out-degree is exactly $0$).
                
                We divide all $2$-cycles into the following exhaustive and mutually exclusive categories:
                \begin{itemize}
                    \item E-B --- exclusively empty branching;
                    \item E-NB --- exclusively empty non-branching;
                    \item NE-NL --- exclusively non-empty non-leaf (could be branching or non-branching);
                    \item NE-L1 --- exclusively non-empty leaf that contains exactly $1$ non-parallel edge in its interior;
                    \item NE-L2 --- exclusively non-empty leaf that contains at least $2$ non-parallel edges in its interior.
                \end{itemize}
                We refer to a $2$-cycle of category $X$, where $X$ is one of the categories above, as an $X$-cycle.

                By definition, every exclusively non-empty $2$-cycle $C$ contains in its exclusive interior at least $1$ edge $e$. By definition of the exclusive interior, that edge $e$ does not belong to the closure of the interior of any of the other $2$-cycles. In particular, $e$ is not an edge of any of the other $2$-cycles. Since $e$ is in the interior of $C$, it cannot be an edge of $C$ itself either. So, $e$ is not an edge of any $2$-cycles, which means that it is a non-parallel edge. Since the exclusive interiors of distinct $2$-cycles do not intersect with each other, we have that these non-parallel edges in the exclusive interiors of distinct exclusively non-empty $2$-cycles are distinct from each other.

                By definition, every E-NB-cycle $C$ has a unique $2$-cycle $D$ exclusively immediately inside it. Therefore, the exclusive interior of $C$ is the region $\operatorname{Int}(C) \setminus \overline{\operatorname{Int}}(D)$ and it is a $4$-face as the only edges incident to it are the two edges of $C$ and the two edges of $D$. We will refer to this $4$-face as the \emph{$4$-face of the exclusive interior of $C$}.

                As was mentioned previously, every leaf $2$-cycle must be exclusively non-empty, so it must contain at least $1$ non-parallel edge in its exclusive interior. So, NE-L1-cycles and NE-L2-cycles together exhaust all leaf $2$-cycles. Also, the exclusive interior of a leaf $2$-cycle is just its interior. The interior of an NE-L1-cycle $C$ is a $4$-face as the only edges incident to it are the two edges of $C$ and the unique edge in its interior that is incident to that face on both sides. We will refer to this $4$-face as the \emph{$4$-face of the interior of $C$}.

                Since the exclusive interiors of distinct $2$-cycles do not intersect with each other, the $4$-faces of the exclusive interiors of E-NB-cycles and the $4$-faces of the interiors of NE-L1-cycles are all pairwise distinct.

                Now, we use discharging. We assign the following initial charges.
                \begin{itemize}
                    \item Each face $f$ has charge $\deg(f)-3$.

                    \item Each parallel edge has charge $-\frac{k+1}{2k}$.

                    \item Each non-parallel edge has charge $1$.

                    \item The pot is empty (has charge $0$).
                \end{itemize}

                We redistribute charges in two stages. The following are the charge redistribution rules of the first stage.
                \begin{itemize}
                    \item[(R1)] Each E-NB-cycle takes charge $1$ from the $4$-face of its exclusive interior.

                    \item[(R2)] Each NE-NL-cycle takes charge $1$ from one arbitrary non-parallel edge in its exclusive interior.
                        
                    \item[(R3)] Each NE-L1-cycle takes charge $1$ from the unique non-parallel edge in its interior and takes charge $1$ from the $4$-face of its interior.

                    \item[(R4)] Each NE-L2-cycle takes charge $1$ from each of two arbitrary non-parallel edges in its interior.
                \end{itemize}

                The following are the charge redistribution rules of the second stage.
                \begin{itemize}
                    \item[(R5)] All $2$-cycles give all their charges to the pot.

                    \item[(R6)] Each parallel edge takes charge $\frac{k+1}{2k}$ from the pot.
                \end{itemize}

                When we say that a $2$-cycle receives or gives charge, we mean that $2$-cycle as a separate entity, not as the two edges that it consists of.

                While the first stage of the charge redistribution is local (in the sense that the charge is redistributed among nearby entities), the second stage is totally global and is carried out through the central pot.

                Regarding the usage of the values $-\frac{k+1}{2k}$ and $\frac{k+1}{2k}$, notice that we can divide by $k$ because we proved earlier that $k \neq 0$.

                Out of the five exhaustive and mutually exclusive categories of $2$-cycles listed above, only E-B-cycles do not receive charges in the first stage. All $2$-cycles of the other four categories receive charges in the first stage, and these four categories correspond to rules (R1), (R2), (R3), (R4).

                Using Lemma~\ref{lemma:planar_multigraph_2m_equal_sum_of_deg_of_faces}, Euler's formula $n-m+\ell=1+p$, the obvious inequality $p \geq 1$, and the previously obtained inequality $3n-3k-7<0$, we calculate the total initial charge as $\sum_f{(\deg(f)-3)}-2k \cdot \frac{k+1}{2k}+(m-2k) \cdot 1+0 = \sum_f{\deg(f)}-\sum_f{3}-(k+1)+(m-2k)=2m-3\ell-3k+m-1=3m-3k-3\ell-1=3m-3k-3(1+p-n+m)-1=3n-3k-4-3p \leq 3n-3k-4-3 \cdot 1=3n-3k-7<0$.

                Clearly, after the charge redistribution, all faces and all edges have non-negative charges. Let us calculate the charge in the pot after the charge redistribution. After the first stage, each leaf $2$-cycle receives charge $2$ either by rule (R3) or by rule (R4). Each unary vertex (that has out-degree exactly $1$) in the rooted directed forest $F$ is either an E-NB-cycle or an NE-NL-cycle and receives charge $1$ either by rule (R1) or by rule (R2), respectively. Each branching $2$-cycle either does not receive any charge or receives charge $1$ by rule (R2), depending on whether it is exclusively empty or not. By Lemma~\ref{lemma:rooted_directed_forest_weight_lower_bound}, the total charge of the $2$-cycles after the first stage is at least $k+1$. Therefore, the pot receives at least $k+1$ of charge by rule (R5). Rule (R6) is applied exactly $2k$ times, hence exactly $2k \cdot \frac{k+1}{2k}=k+1$ of charge is taken from the pot. Therefore, the charge in the pot after the charge redistribution is at least $k+1-(k+1)=0$. So, the charge in the pot after the charge redistribution is also non-negative. This implies that the total charge after the charge redistribution is non-negative, which contradicts the negative total initial charge.
            \end{proof}

            As mentioned above, there is an alternative way of dealing with parallel edges with multiplicity greater than $2$ without assuming the edge-minimality of a counterexample and without Lemma~\ref{lemma:without_2-faces_min_counterexample_parallel_edges_multiplicity_geq_3_reducible}. Namely, we can consider a set $S$ of $2$-cycles where, for each pair of vertices with parallel edges between them, we arbitrarily choose one $2$-cycle. Then the number of $2$-cycles in $S$ would be exactly $k$. Then, instead of defining the exclusivity (exclusive interior, exclusively empty $2$-cycle) as above, we can define \emph{$S$-exclusivity} where the requirements are applied only to $2$-cycles from $S$ instead of to all $2$-cycles. Specifically, for a $2$-cycle $C$, we can define the \emph{$S$-exclusive interior} of $C$ as $\operatorname{Int}(C) \setminus \left(\cup_{D \in S, D \neq C}{\overline{\operatorname{Int}}(D)}\right)$. We can define \emph{$S$-edges} as edges belonging to $2$-cycles from $S$ and \emph{non-$S$-edges} as all other edges, not belonging to $2$-cycles from $S$. All $S$-edges are parallel, but non-$S$-edges could include both non-parallel ones and parallel ones. We would treat $S$-edges and non-$S$-edges the same as we treat parallel ones and non-parallel ones in the current proof. For example, in the discharging, each $S$-edge would get the initial charge $-\frac{k+1}{2k}$, and each non-$S$-edge would get the initial charge $1$. That approach could be useful for a potential future proof of an improved lower bound where we need the edge-\emph{maximality} of a vertex-minimal counterexample for other purposes and thus cannot use the edge-\emph{minimality}. However, in the current proof, we use the simpler approach of assuming the edge-minimality and using Lemma~\ref{lemma:without_2-faces_min_counterexample_parallel_edges_multiplicity_geq_3_reducible}, without introducing $S$-exclusivity.

        \subsection{The construction}
            \begin{lemma}\label{lemma:disjoint_union_sum_a}
                Let $M$ be a multigraph that is a disjoint union of multigraphs $M_1$, \ldots, $M_q$, where $q \geq 1$. Then we have $a(M)=\sum_{i=1}^{q}{a(M_i)}$.
            \end{lemma}
            \begin{proof}
                First, we prove that $a(M) \leq \sum_{i=1}^{q}{a(M_i)}$. Let $F$ be a maximum induced forest in $M$. Since any subgraph of a forest is a forest itself, the graph $M_i[V(F) \cap V(M_i)]$ is an induced forest in $M_i$ for each $i$. Therefore, we have $a(M_i) \geq |V(F) \cap V(M_i)|$. Now, we have $a(M) = |V(F)| = \sum_{i=1}^{q}{|V(F) \cap V(M_i)|} \leq \sum_{i=1}^{q}{a(M_i)}$.

                Conversely, for each graph $M_i$, let $F_i$ be a maximum induced forest in $M_i$. We consider the induced subgraph $F'=M\left[\bigcup_{i=1}^{q}{V(F_i)}\right]$. Since there are no edges between different $M_i$, $F'$ is an induced forest in $M$. Therefore, we have $a(M) \geq |V(F')| = \sum_{i=1}^{q}{|V(F_i)|} = \sum_{i=1}^{q}{a(M_i)}$.

                Combining the two obtained opposite inequalities, we get the claimed equality $a(M)=\sum_{i=1}^{q}{a(M_i)}$.
            \end{proof}

            \begin{theorem}\label{theorem:without_2-faces_example_3n/7+4/7}
                There exists an infinite sequence of planar multigraphs admitting an embedding into the plane without $2$-faces with $a(M)=\frac{3}{7}n+\frac{4}{7}$, where $n$ is the number of vertices in $M$.
            \end{theorem}
            \begin{proof}
                We define the sequence of plane multigraphs $\mathcal{N}_i$ recursively. Take $\mathcal{N}_1=\mathcal{K}_4$, where by $\mathcal{K}_4$ we denote the plane graph that is an embedding of $K_4$ into the plane. Then, for each $i$, take three vertices $u_i$, $v_i$, $w_i$. Connect $u_i$ and $v_i$ with two parallel edges. Put $w_i$ in the interior of the $2$-cycle formed by these parallel edges between $u_i$ and $v_i$. Connect $u_i$ and $w_i$ with two parallel edges. Connect $v_i$ and $w_i$ with two parallel edges. Put a copy of $\mathcal{K}_4$, disconnected from the rest of the multigraph, in the interior of the $2$-cycle formed by the parallel edges between $u_i$ and $w_i$. Put a copy of $\mathcal{N}_i$, disconnected from the rest of the multigraph, in the interior of the $2$-cycle formed by the parallel edges between $v_i$ and $w_i$. Denote the resulting plane multigraph by $\mathcal{N}_{i+1}$. Finally, we consider a plane multigraph $\mathcal{M}_i$ that is obtained from $\mathcal{N}_i$ by adding a copy of $\mathcal{K}_4$, disconnected from the rest of the multigraph, in the external $2$-face.

                \begin{figure}[htb]
                    \centering
                    \setlength{\unitlength}{0.8mm}
                    \begin{picture}(72,50)(15,25)
                        \qbezier(20,50)(50,100)(80,50)
                        \qbezier(20,50)(50,0)(80,50)
                        
                        \qbezier(20,50)(35,75)(50,50)
                        \qbezier(20,50)(35,24)(50,50)
                        
                        \qbezier(50,50)(65,75)(80,50)
                        \qbezier(50,50)(65,25)(80,50)
                        
                        \put(20,50){\circle*{1.5}}  
                        \put(80,50){\circle*{1.5}}  
                        \put(50,50){\circle*{1.5}}  
                        
                        \put(16,50){\makebox(0,0){$u_i$}}
                        \put(84,50){\makebox(0,0){$v_i$}}
                        \put(50,56){\makebox(0,0){$w_i$}}
                        
                        \put(35,59){\circle*{1.5}}  
                        \put(28,45){\circle*{1.5}}  
                        \put(42,45){\circle*{1.5}}  
                        \put(35,50){\circle*{1.5}}  
                        
                        \Line(35,59)(28,45)
                        \Line(35,59)(42,45)
                        \Line(28,45)(42,45)
                        \Line(35,50)(35,59)
                        \Line(35,50)(28,45)
                        \Line(35,50)(42,45)
                        
                        \put(65,50){\makebox(0,0){$\mathcal{N}_i$}}
                    \end{picture}
                    \caption{$\mathcal{N}_{i+1}$ in the proof of Theorem~\ref{theorem:without_2-faces_example_3n/7+4/7}.}
                    \label{picture:without_2-faces_example_3n/7+4/7_N_i+1}
                \end{figure}

                Let us prove using induction on $i$ that none of the internal faces of $\mathcal{N}_i$ are $2$-faces and that the external face of $\mathcal{N}_i$ is a $3$-face for $i=1$ and a $2$-face for $i \geq 2$. To verify the base of induction for $i=1$, we observe that all four faces of $\mathcal{N}_1=\mathcal{K}_4$ are $3$-faces. Suppose that the statement holds for $\mathcal{N}_i$ and consider it for $\mathcal{N}_{i+1}$, where $i \geq 1$. By construction, the external face of $\mathcal{N}_{i+1}$ is incident only to the two parallel edges between $u_i$ and $v_i$, and hence it is a $2$-face. The vertex $w_i$ and the four edges connecting it to $u_i$ and $v_i$ break the interior of the $2$-cycle between $u_i$ and $v_i$ into four regions: the interior of the $2$-cycle between $u_i$ and $w_i$ that will be broken down further by the added copy of $\mathcal{K}_4$, the interior of the $2$-cycle between $v_i$ and $w_i$ that will be broken down further by the added copy of $\mathcal{N}_i$, a $3$-face that is incident to one of the parallel edges between $u_i$ and $v_i$, one of the parallel edges between $u_i$ and $w_i$, and one of the parallel edges between $v_i$ and $w_i$, and another $3$-face that is incident to the other edge between $u_i$ and $v_i$, the other edge between $u_i$ and $w_i$, and the other edge between $v_i$ and $w_i$. The interior of the $2$-cycle between $u_i$ and $w_i$ is broken by the added copy of $\mathcal{K}_4$ into four faces: three internal $3$-faces of the copy of $\mathcal{K}_4$ and one $5$-face that is incident to the three external edges of the copy of $\mathcal{K}_4$ and to the two parallel edges between $u_i$ and $w_i$. The interior of the $2$-cycle between $v_i$ and $w_i$ is broken by the added copy of $\mathcal{N}_i$ into the internal faces of the copy of $\mathcal{N}_i$, none of which are $2$-faces by the inductive hypothesis, and the remaining face that is incident to the external edges of the copy of $\mathcal{N}_i$ and to the two parallel edges between $v_i$ and $w_i$. That last remaining face is a $5$-face for $i=1$ because there are $3$ external edges in $\mathcal{N}_1$, and a $4$-face for $i \geq 2$ because there are two external edges in the copy of $\mathcal{N}_i$ by the inductive hypothesis. So, none of the internal faces of $\mathcal{N}_{i+1}$ are $2$-faces, which completes the inductive step.

                Now, for $i \geq 2$, the region of the plane that was the external $2$-face in $\mathcal{N}_i$ is broken by the added copy of $\mathcal{K}_4$ into four faces in $\mathcal{M}_i$: three internal $3$-faces of the copy of $\mathcal{K}_4$ and one $5$-face that is incident to the three external edges of the copy of $\mathcal{K}_4$ and to the two external parallel edges of $\mathcal{N}_i$. None of the other faces of $\mathcal{N}_i$ are $2$-faces, and they all are unchanged in $\mathcal{M}_i$. So, $\mathcal{M}_i$ does not contain any $2$-faces.

                Denote by $N_i$ and $M_i$ the underlying planar multigraphs of the plane multigraphs $\mathcal{N}_i$ and $\mathcal{M}_i$, respectively. Denote by $n_i'$ and $n_i$ the number of vertices in $N_i$ and $M_i$, respectively, and denote $a_i'=a(N_i)$ and $a_i=a(M_i)$. It is easy to calculate that, for each $i$, we have $n_{i+1}'=4+n_i'+3$. Observe that $N_{i+1}$ is a disjoint union of $K_4$, $N_i$, and $N_{i+1}[u_i,v_i,w_i]$. It can be directly verified that $a(K_4)=2$. Since the vertices $u_i$, $v_i$, $w_i$ have a pair of parallel edges between every pair of them, no more than one of them can belong to an induced forest in $N_{i+1}[u_i,v_i,w_i]$. On the other hand, any single one of them is an induced forest in $N_{i+1}[u_i,v_i,w_i]$. So, we have $a(N_{i+1}[u_i,v_i,w_i])=1$. Now, by Lemma~\ref{lemma:disjoint_union_sum_a}, we have $a_{i+1}'=a(K_4)+a(N_i)+a(N_{i+1}[u_i,v_i,w_i])=2+a_i'+1$. Using induction, we derive that $n_i'=4+7(i-1)=7i-3$ and $a_i'=2+3(i-1)=3i-1$.

                Finally, we have $n_i=n_i'+4=7i+1$. Observe that $M_i$ is a disjoint union of $N_i$ and $K_4$. By Lemma~\ref{lemma:disjoint_union_sum_a}, we have $a_i=a(N_i)+a(K_4)=a_i'+2=(3i-1)+2=3i+1$. Therefore, we have $a_i=3i+1=3\frac{n_i-1}{7}+1=\frac{3}{7}n_i+\frac{4}{7}$, which means that the sequence of multigraphs $\{M_i\}$ satisfies the claimed condition.
            \end{proof}

            It is natural to ask whether there exists a connected construction with the same $a(M)$. The answer is affirmative. We will now show how to adjust the construction from Theorem~\ref{theorem:without_2-faces_example_3n/7+4/7} to make it connected. The main idea is to prove a stronger version of Lemma~\ref{lemma:disjoint_union_sum_a} that would guarantee the additivity of $a(M)$ under weaker assumptions that do not require the parts to be disconnected from each other.

            If a multigraph $M$ is vertex-partitioned into induced subgraphs $M_1$, \ldots, $M_q$, then we consider a multigraph $H$, which we call the \emph{contraction multigraph of the partition}, with the vertices $M_1$, \ldots, $M_q$ and an edge between $M_i$ and $M_j$ for $i \neq j$ for each edge in $M$ between vertices of $M_i$ and vertices of $M_j$. In other words, the contraction multigraph of the partition is formed by contracting each part of the partition into a single vertex and removing the resulting loops (but keeping the resulting parallel edges).

            \begin{lemma}\label{lemma:forest_union_cycle_inside_single_part}
                Let $M$ be a multigraph that is vertex-partitioned into induced subgraphs $M_1$, \ldots, $M_q$. Assume that the contraction multigraph of the partition is a forest (and, in particular, it is simple). Then every cycle in $M$ is contained entirely inside a single part of the partition.
            \end{lemma}
            \begin{proof}
                Suppose, to the contrary, that there exists a cycle $C$ in $M$ that is not contained entirely inside any single part of the partition. Then it contains vertices of at least $2$ distinct parts. Denote the contraction multigraph of the partition by $H$. It is easy to see that, when we contract each part into a single vertex and remove the resulting loops, the cycle $C$ in $M$ becomes a closed walk in $H$ that we denote by $W$. Since $C$ contains vertices of at least $2$ distinct parts, $W$ contains at least $2$ vertices, which implies that its length is at least $2$. Since $H$ is a forest and thus does not contain any cycles, $W$ must traverse at least $1$ edge twice in opposite directions. This implies that $C$ must traverse the same edge twice in opposite directions as well, but that is impossible for a cycle.
            \end{proof}
            
            \begin{lemma}\label{lemma:forest_union_sum_a}
                Let $M$ be a multigraph that is vertex-partitioned into induced subgraphs $M_1$, \ldots, $M_q$. Assume that the contraction multigraph of the partition is a forest (and, in particular, it is simple). Then we have $a(M)=\sum_{i=1}^{q}{a(M_i)}$.
            \end{lemma}
            \begin{proof}
                First, we prove that $a(M) \leq \sum_{i=1}^{q}{a(M_i)}$. Let $F$ be a maximum induced forest in $M$. Since any subgraph of a forest is a forest itself, the graph $M_i[V(F) \cap V(M_i)]$ is an induced forest in $M_i$ for each $i$. Therefore, we have $a(M_i) \geq |V(F) \cap V(M_i)|$. Now, we have $a(M) = |V(F)| = \sum_{i=1}^{q}{|V(F) \cap V(M_i)|} \leq \sum_{i=1}^{q}{a(M_i)}$.

                Conversely, for each graph $M_i$, let $F_i$ be a maximum induced forest in $M_i$. We consider the induced subgraph $F'=M\left[\bigcup_{i=1}^{q}{V(F_i)}\right]$. Let us prove that $F'$ is a forest. Suppose, to the contrary, that $F'$ contains a cycle. Then, by Lemma~\ref{lemma:forest_union_cycle_inside_single_part}, that cycle is contained entirely inside a single part $M_i$ of the partition. This implies that it is contained in $F_i$, which contradicts the fact that $F_i$ is a forest. Therefore, $F'$ is an induced forest. Therefore, we have $a(M) \geq |V(F')| = \sum_{i=1}^{q}{|V(F_i)|} = \sum_{i=1}^{q}{a(M_i)}$.

                Combining the two obtained opposite inequalities, we get the claimed equality $a(M)=\sum_{i=1}^{q}{a(M_i)}$.
            \end{proof}

            Lemma~\ref{lemma:disjoint_union_sum_a} is a particular case of Lemma~\ref{lemma:forest_union_sum_a} when the contraction multigraph of the partition is edgeless.

            \begin{theorem}\label{theorem:without_2-faces_example_3n/7+4/7_connected}
                There exists an infinite sequence of connected planar multigraphs admitting an embedding into the plane without $2$-faces with $a(M)=\frac{3}{7}n+\frac{4}{7}$, where $n$ is the number of vertices in $M$.
            \end{theorem}
            \begin{proof}
                We adjust the construction from Theorem~\ref{theorem:without_2-faces_example_3n/7+4/7} by adding two additional edges in $\mathcal{N}_{i+1}$: one between $u_i$ and one of the external vertices of the copy of $\mathcal{K}_4$ and the other between $v_i$ and one of the external vertices of the copy of $\mathcal{N}_i$, and an additional edge in $\mathcal{M}_i$ between one of the external vertices of the copy of $\mathcal{N}_i$ and one of the external vertices of the copy of $\mathcal{K}_4$.

                \begin{figure}[htb]
                    \centering
                    \setlength{\unitlength}{0.8mm}
                    \begin{picture}(72,50)(15,25)
                        \qbezier(20,50)(50,100)(80,50)
                        \qbezier(20,50)(50,0)(80,50)
                        
                        \qbezier(20,50)(35,75)(50,50)
                        \qbezier(20,50)(35,24)(50,50)
                        
                        \qbezier(50,50)(65,75)(80,50)
                        \qbezier(50,50)(65,25)(80,50)
                        
                        \put(20,50){\circle*{1.5}}  
                        \put(80,50){\circle*{1.5}}  
                        \put(50,50){\circle*{1.5}}  
                        
                        \put(16,50){\makebox(0,0){$u_i$}}
                        \put(84,50){\makebox(0,0){$v_i$}}
                        \put(50,56){\makebox(0,0){$w_i$}}
                        
                        \put(35,59){\circle*{1.5}}  
                        \put(28,45){\circle*{1.5}}  
                        \put(42,45){\circle*{1.5}}  
                        \put(35,50){\circle*{1.5}}  
                        
                        \Line(35,59)(28,45)
                        \Line(35,59)(42,45)
                        \Line(28,45)(42,45)
                        \Line(35,50)(35,59)
                        \Line(35,50)(28,45)
                        \Line(35,50)(42,45)
                        
                        \put(65,50){\makebox(0,0){$\mathcal{N}_i$}}

                        \Line(20,50)(28,45)
                        \Line(80,50)(68.5,50)
                    \end{picture}
                    \caption{$\mathcal{N}_{i+1}$ in the proof of Theorem~\ref{theorem:without_2-faces_example_3n/7+4/7_connected}.}
                    \label{picture:without_2-faces_example_3n/7+4/7_connected_N_i+1}
                \end{figure}

                It is easy to see that $\mathcal{N}_{i+1}$ is vertex-partitioned into $\mathcal{N}_{i+1}[u_i,v_i,w_i]$, the copy of $\mathcal{K}_4$, and the copy of $\mathcal{N}_i$. The contraction multigraph of that partition contains exactly two edges: between $\mathcal{N}_{i+1}[u_i,v_i,w_i]$ and the other two parts (they correspond to the two additional edges in our adjustment of the construction). So, that contraction multigraph of the partition is a path of length $2$ and hence is a forest. Therefore, Lemma~\ref{lemma:forest_union_sum_a} is applicable to that partition.

                Similarly, $\mathcal{M}_i$ is vertex-partitioned into the copy of $\mathcal{N}_i$ and the copy of $\mathcal{K}_4$ with a single edge between the two parts. Therefore, the contraction multigraph of that partition contains exactly one edge, which means that it is a path of length $1$ and hence is a forest. Therefore, Lemma~\ref{lemma:forest_union_sum_a} is applicable to that partition.

                The rest of the proof follows exactly the proof of Theorem~\ref{theorem:without_2-faces_example_3n/7+4/7}, using Lemma~\ref{lemma:forest_union_sum_a} instead of Lemma~\ref{lemma:disjoint_union_sum_a} and adjusting the degrees of the faces in the verification that the construction does not contain $2$-faces (there would be $7$-faces instead of the $5$-faces and a $6$-face instead of the $4$-face because each of the additional edges does not break the face where it was added into multiple faces, but adds $2$ to the degree of that face).
            \end{proof}

        \subsection{Concluding remarks and open questions}
            From Theorem~\ref{theorem:without_2-faces_lower_bound_3n/10+7/30} and Theorem~\ref{theorem:without_2-faces_example_3n/7+4/7}, the best possible lower bound on $a(M)$ for the variant without $2$-faces must be somewhere between $\frac{3}{10}n+\frac{7}{30}$ and $\frac{3}{7}n+\frac{4}{7}$. We do not know the best possible lower bound. Nor do we even know whether the exact value $\frac{3}{7}n$ (without any additive constant) is attained. We leave these as open questions.

            \begin{question}
                Does there exist a planar multigraph $M$ admitting an embedding into the plane without $2$-faces with $a(M) \leq \frac{3}{7}n$?
            \end{question}
            
            \begin{question}
                What is the best possible lower bound on $a(M)$ in terms of $n$ for every planar multigraph $M$ on $n$ vertices that admits an embedding into the plane without $2$-faces?
            \end{question}

    \printbibliography
\end{document}